\documentclass[10pt]{amsart}
\usepackage{amssymb}
\usepackage{hyperref}

\newtheorem{theorem}{Theorem}[section]
\newtheorem{lemma}[theorem]{Lemma}
\newtheorem{corollary}[theorem]{Corollary}
\newtheorem{proposition}[theorem]{Proposition}
\newtheorem{conjecture}[theorem]{Conjecture}

\theoremstyle{definition}

\newtheorem{example}[theorem]{Example}

\theoremstyle{remark}
\newtheorem{remark}[theorem]{Remark}

\numberwithin{equation}{section}
\newenvironment{acknowledgement}{\noindent \textit{Acknowledgement.}}{}

\begin{document}
\title[Heat Kernel Traces on Foliations]{Traces of heat operators on
Riemannian foliations}
\author{Ken Richardson}
\address{Ken Richardson\\
Department of Mathematics\\
Texas Christian University\\
TCU Box 298900\\
Fort Worth, Texas 76129}
\email{k.richardson@tcu.edu}
\date{October, 2007}
\thanks{Research at MSRI is supported in part by NSF grant DMS-9701755.}

\begin{abstract}
We consider the basic heat operator on functions on a Riemannian foliation
of a compact, Riemannian manifold, and we show that the trace $K_{B}(t) $ of
this operator has a particular asymptotic expansion as $t\to 0$. The
coefficients of $t^{\alpha}$ and of $t^{\alpha}(\log t)^{\beta}$ in this
expansion are obtainable from local transverse geometric invariants -
functions computable by analyzing the manifold in an arbitrarily small
neighborhood of a leaf closure. Using this expansion, we prove some results
about the spectrum of the basic Laplacian, such as the analogue of Weyl's
asymptotic formula. Also, we explicitly calculate the first two nontrivial
coefficients of the expansion for special cases such as codimension two
foliations and foliations with regular closure.
\end{abstract}

\maketitle





\section{Introduction}

Let $M$ be an $n$-dimensional, compact, connected, oriented Riemannian
manifold without boundary. The heat kernel is the fundamental solution to
the associated heat equation. That is, it is the unique function $%
K:(0,\infty )\times M\times M$ that satisfies 
\begin{eqnarray*}
\left( \frac{\partial }{\partial t}+\Delta _{x}\right) K(t,x,y) &=&0\qquad 
\text{ and} \\
\lim_{t\rightarrow 0^{+}}\int_{M}K(t,x,y)\,f(y)\,\mathrm{dvol}(y)
&=&f(x)\qquad \text{for every continuous function }f.
\end{eqnarray*}
This function $K$ can be used to solve the heat equation $\left( \frac{%
\partial }{\partial t}+\Delta _{x}\right) g(t,x)=0$ for any initial
temperature distribution $g(0,x)$. It is well known (\cite{MiP}; see also 
\cite{Cha},\cite{Ro}) that for any $x\in M$ and any positive integer $k$, 
\begin{equation}
K(t,x,x)=\frac{1}{(4\pi t)^{n/2}}\left( u_{0}(x)+u_{1}(x)t+\dots
+u_{k}(x)t^{k}+O\left( t^{k+1}\right) \right) \text{ as }t\rightarrow 0,
\label{ktxx}
\end{equation}
where $u_{j}(x)$ are smooth functions on $M$ that depend only on geometric
data at the point $x\in M$. In particular, $u_{0}(x)=1$ and $u_{1}(x)=\frac{%
S(x)}{6}$, where $S(x)$ is the scalar curvature of $M$ at $x$. Using the
expansion above, it is possible to prove that the trace of the heat kernel
has a similar asymptotic formula. Let $0=\lambda _{0}<\lambda _{1}\leq
\lambda _{2}\leq \lambda _{3}\dots $ be the eigenvalues of the Laplacian,
counting multiplicities. Then 
\begin{eqnarray}
\text{tr}\,\left( e^{-t\Delta }\right) &=&\sum_{m\geq 0}e^{-t\lambda
_{m}}=\int_{M}K(t,x,x)\,\mathrm{dvol}(x)  \notag \\
&=&\frac{1}{(4\pi t)^{n/2}}\left( U_{0}+U_{1}t+\dots +U_{k}t^{k}+O\left(
t^{k+1}\right) \right) ,  \label{trace}
\end{eqnarray}
where $U_{j}=\int_{M}u_{j}(x)\,\mathrm{dvol}(x)$, with $u_{j}(x)$ defined as
above. In particular, $U_{0}=\text{Vol}(M)$. From formula~(\ref{trace}),
Karamata's theorem (see, for example, \cite[pp.~418--423]{Fel}) implies the
Weyl asymptotic formula ([We]; see also \cite[p.~155]{Cha}): 
\begin{eqnarray*}
N(\lambda ):= &\#\{\lambda _{m}|\lambda _{m}\leq &\lambda \} \\
&\sim &\frac{\text{Vol}(M)}{(4\pi )^{n/2}\Gamma \left( \frac{n}{2}+1\right) }%
\lambda ^{n/2}
\end{eqnarray*}
as $\lambda \rightarrow \infty $.

The heat kernel has also been studied more generally, such as in the case of
manifolds with boundary or in the case of elliptic operators acting on
sections of a vector bundle over the manifold. Many researchers have studied
this expansion and its generalizations and have worked to compute the
coefficient functions (see \cite{MiP},\cite{Be},\cite{McS},\cite{Gi}),
because the heat kernel is not only used to compute heat flows but is also
used in many areas of geometric and topological analysis. The asymptotic
expansions above (and their generalizations) have been used to study the
spectrum of the Laplacian (see \cite{Be},\cite{BeGM},\cite{Cha},\cite{McS}),
the determinant of the Laplacian (see \cite{OPS1},\cite{Ri1}), conformal
classes of metrics (see \cite{OPS2}), analytic torsion (see \cite{RaS},\cite%
{Che}), modular forms (see \cite{Feg}), index theory (see \cite{ABP},\cite%
{Ro}), stochastic analysis (see \cite{CarZ},\cite{MaSt}), gauge
theory/mathematical physics (see \cite{EFKK}, \cite{Cam},\cite{Bl}), and so
on.

Other researchers have studied generalizations of the heat kernel to orbit
spaces of a group acting on a manifold. In \cite{D2}, the author showed that
if $M$ is a connected $n$--dimensional (not necessarily compact) Riemannian
manifold and $\Gamma $ is a group acting isometrically, effectively, and
properly discontinuously on $M$ with compact quotient $\overline{M}$, the
induced heat operator $e^{-t\overline{\Delta }}$ on the space of functions
on $\overline{M}$ (which is not necessarily a manifold) satisfies 
\begin{equation*}
\text{tr}\,\left( e^{-t\overline{\Delta }}\right) =\frac{1}{(4\pi t)^{n/2}}%
\left( \overline{U}_{0}+\overline{U}_{1}t+\ldots +\overline{U}%
_{k}t^{k}+O\left( t^{k+1}\right) \right) ,
\end{equation*}
where $n=\mathrm{dim}(M)$ and $\overline{U}_{0}=\mathrm{Vol}(\overline{M})$.
This is equivalent to calculating the trace of the ordinary heat kernel on $%
M $ restricted to $\Gamma $-invariant functions. In \cite{BrH2}, the
researchers considered a compact, $n$-dimensional Riemannian manifold $M$
along with a compact group $G$ of isometries. Let $E_{\lambda }$ denote the
complex eigenspace of $\Delta $ associated to the eigenvalue $\lambda $, and
let $E_{\lambda }^{G}$ denote the subspace of $E_{\lambda }$ consisting of
eigenfunctions invariant under the induced action of $G$. They show that the
associated equivariant trace for $t>0$ is

\begin{eqnarray}
L(t):= &&\sum_{\lambda \geq 0}e^{-\lambda t}\dim E_{\lambda }^{G}  \notag \\
&\sim &(4\pi t)^{-m/2}\left( a_{0}+\sum_{j>k\geq 0}a_{jk}t^{j/2}(\log
t)^{k}\right) \text{ as }t\rightarrow 0,  \label{equtrace}
\end{eqnarray}%
where $m=\dim M/G$, $K_{0}$ is less than or equal to the number of different
dimensions of $G$-orbits in $M$, and $a_{0}=$Vol$(M/G)$. The coefficients $%
a_{jk}$ depend only on the metrics on $M$ and $G$ and their derivatives on
the subset $\{(g,x)\,|\,xg=x\}\subset G\times M$. The authors show in
addition that under certain conditions, no logarithmic terms appear in the
asymptotic expansion. Clearly, no logarithmic terms occur if all of the
orbits have the same dimension. Also, if $G$ is connected of rank 1 and acts
effectively on $M$, then no logarithmic terms appear. We remark that the
results of \cite{BrH2} apply to more general situations. If a second order
differential operator has the same principal symbol as the Laplacian, and if
it is geometrically defined and thus commutes with the action of $G$, then
the equivariant trace of the corresponding heat kernel satisfies~(\ref%
{equtrace}). After writing this paper, 
I was made aware of the recent work \cite{Dry}, where the authors
compute the asymptotics of the heat kernel on orbifolds, related to the work in 
\cite{D2}, \cite{BrH2}, and to Theorem~\ref{orbTheorem} in this paper.

In this paper, we consider a generalization of the trace of the heat kernel
to Riemannian foliations, and we will observe asymptotic behavior similar to
the results for group actions. Suppose that a compact, Riemannian manifold $M
$ is equipped with a \textit{Riemannian foliation} $\mathcal{F}$; that is,
the distance from one leaf of $\mathcal{F}$ to another is locally constant.
For simplicity, we assume that $M$ is connected and oriented and that the
foliation is transversally oriented. In some sense, this is a generalization
of the work in \cite{BrH2} and \cite{D2}, because the orbits of a group
acting by isometries form an example of a Riemannian foliation, if the
orbits all have the same dimension. In \cite{CrP}, the authors explicitly
calculated the heat kernel expansion in this specific case. Of course, the
dimensions of orbits of arbitrary group actions on a manifold are typically
not constant, and the leaf closures of a foliation are generally not orbits
of a group action. In \cite{Ri3}, we showed that many problems in the
analysis of the transverse geometry of Riemannian foliations and that of
group actions are equivalent problems.

A natural question to consider is the following: if we assume that the
temperature is always constant along the leaves of $\left( M,\mathcal{F}%
\right) $, how does heat flow on the manifold? To answer this question, we
must restrict to the space of basic functions $C_{B}^{\infty }(M)$(those
that are constant on the leaves of the foliation) and more generally the
space of basic forms $\Omega _{B}^{\ast }(M)$(smooth forms $\omega $such
that given any vector $X$tangent to the leaves, $i(X)\omega =0$and $%
i(X)d\omega =0$, where $i(X)$denotes the interior product with $X$). The
exterior derivative $d$maps basic forms to basic forms; let $d_{B}$denote $d$%
restricted to $\Omega _{B}^{\ast }(M)$. The relevant Laplacian on forms is
the basic Laplacian $\Delta _{B}=d_{B}\delta _{B}+\delta _{B}d_{B}$, where $%
\delta _{B}$is the adjoint of $d_{B}$on $L^{2}\left( \Omega _{B}^{\ast
}(M)\right) $. The basic heat kernel $K_{B}(t,x,y)$on functions is a
function on $(0,\infty )\times M\times M$that is basic in each $M$factor and
that satisfies 
\begin{eqnarray}
\left( \frac{\partial }{\partial t}+\Delta _{B,x}\right) K_{B}(t,x,y) &=&0 
\notag \\
\lim_{t\rightarrow 0^{+}}\int_{M}K_{B}(t,x,y)\,f(y)\,\mathrm{dvol}(y) &=&f(x)
\label{basicdef}
\end{eqnarray}
for every continuous basic function $f$. The existence of the basic heat
kernel allows us to answer the question posed at the beginning of this
paragraph. The basic heat kernel on forms is defined in an analogous way.
Many researchers have studied the analytic and geometric properties of the
basic Laplacian and the basic heat kernel (see \cite{A1},\cite{E},%
\cite{EH},\cite{KT},\cite{NRT}, \cite{NTV},\cite{PaRi}). In \cite{E}, the
author proved the existence of the basic heat kernel on functions. The
existence of the basic heat kernel on forms was proved for the case where
the mean curvature form of the foliation is basic in \cite{NRT}. The
existence of the basic heat kernel was proved in general in \cite{PaRi},
where the authors given explicit formulas for the basic Laplacian and basic
heat kernel in terms of the orthogonal projection from $L^{2}$--forms to $%
L^{2}$--basic forms and certain elliptic operators on the space of all forms
on the manifold. A point of difficulty that often arises in this area of
research is that the space of basic forms is not the set of all sections of
any vector bundle, and therefore the usual theory of elliptic operators and
heat kernels does not apply directly to $\Delta _{B}$ and $K_{B}$.

It is natural to try to prove the existence of asymptotic expansions of the
form (\ref{ktxx}) and (\ref{trace}) for the basic heat kernel. We remark
that the basic heat operator is trace class, since the basic Laplacian is
the restriction of an elliptic operator on the space of all functions (see
lower bounds for eigenvalues in \cite{PaRi} and \cite{LRi}). In \cite{Ri2},
it was shown that an analogue of (\ref{ktxx}) exists for the basic heat
kernel. As $t\rightarrow 0$, we have the following asymptotic expansion for
any positive integer $k$: 
\begin{equation}  \label{kbtxx}
K_{B}(t,x,x)=\frac{1}{(4\pi t)^{{}^{q_{x}/2}}}\left(
a_{0}(x)+a_{1}(x)t+\ldots +a_{k}(x)t^{k}+O\left( t^{k+1}\right) \,\right) ,
\end{equation}
where $q_{x}$ is the codimension of the leaf closure containing $x$ and $%
a_{j}(x)$ are functions depending on the local transverse geometry and
volume of the leaf closure containing $x$. The first two nontrivial
coefficients were computed in \cite{Ri2} and are given in Theorem~\ref%
{oldthm}. In general, the power $q_{x}$ may vary, but its value is minimum
and constant on an open, dense subset of $M$. One might guess that the
asymptotics of the trace of the basic heat operator could be obtained by
integrating the expansion (\ref{kbtxx}), similar to obtaining (\ref{trace})
from (\ref{ktxx}). However, the functions $a_{j}(x)$ for $j\geq 1$ are not
necessarily bounded or even integrable over the dense subset. Example~\ref%
{sphereexample} exhibits this precise behavior. Even if the coefficients $%
a_{j}(x)$ in (\ref{kbtxx}) are integrable, it is not true in general that
these functions can be integrated to obtain the asymptotics of the trace.
Example~\ref{exnonorientable} shows that even if $a_{j}(x)$ is constant,
these coefficients cannot be integrated to yield the trace asymptotics.

Despite these obstacles, we prove that an asymptotic expansion for the trace
of the basic heat operator exists. Let $\overline{q}$ be the minimum
codimension of the leaf closures of $(M,\mathcal{F})$. As $t\rightarrow 0$,
the trace $K_{B}(t)$ of the basic heat kernel on functions satisfies the
following asymptotic expansion for any positive integer $J$: 
\begin{equation}
K_{B}(t)=\frac{1}{(4\pi t)^{\overline{q}/2}}\left( a_{0}+\sum_{j>0,~0\leq
k\leq K_{0}}a_{jk}t^{j/2}(\log t)^{k}+O\left( t^{\frac{J+1}{2}}(\log
t)^{K_{0}-1}\right) \right) ,  \label{tb}
\end{equation}%
where $K_{0}$ is less than or equal to the number of different dimensions of
leaf closures in $\mathcal{F}$, and where 
\begin{equation*}
a_{0}=V_{tr}=\int_{M}\frac{1}{\text{Vol}\left( \overline{L_{x}}\right) }\,{%
\mathrm{dvol}}(x).
\end{equation*}%
This is the content of Theorem~\ref{tracebasic}. If the codimension of $%
\mathcal{F}$ is less than $4$, then the logarithmic terms vanish. The idea
of proof is as follows. We rewrite the integral $\displaystyle %
K_{B}(t)=\int_{M}K_{B}(t,x,x)\,\mathrm{dvol}(x)$ in terms of an integral
over $W\times SO(q)$, where $W$ is the \textit{basic manifold}, an $SO(q)$%
-manifold associated to $(M,\mathcal{F})$. Then, we apply the results of 
\cite{BrH2}. In Corollary~\ref{Weyl}, we obtain the Weyl asymptotic formula
for the eigenvalues of the basic Laplacian.

In Section~\ref{special}, we derive the first two nontrivial coefficients in
the asymptotic expansion (\ref{tb}) in some special cases, including but not
limited to all possible types of Riemannian foliations of codimension two or
less. In each of these cases, the asymptotic formula contains no logarithmic
terms. We conjecture (Conjecture~\ref{traceconjecture}) that the asymptotic
expansion for the general case has the same features. In Section~\ref{finite}%
, we derive the asymptotics for the case in which all of the leaf closures
have the same dimension, for any codimension. In this case, the leaf closure
space is an orbifold, and en route to the result, we obtain the asymptotics
of the orbifold heat trace, which may be of independent interest. We find
the asymptotics for the transversally orientable, codimension one case in
Section~\ref{codim1}, for the nonorientable codimension one case in Section~%
\ref{nonorientable}, and for codimension two Riemannian foliations in
Section~\ref{codim2}. We remark that the codimension two case yields five
possible types of asymptotic expansions. In Section~\ref{general}, we show
how to simplify the general case by subdividing the basic manifold into
pieces, and this result is used in the calculations of Section~\ref{codim2}.
In Section~\ref{example}, we demonstrate the asymptotic formulas in two
examples of codimension two foliations.

We remark that these asymptotic expansions yield new results concerning the
spectrum of the basic Laplacian. By Corollary~\ref{Weyl}, the eigenvalues of
the basic Laplacian determine the minimum leaf closure codimension and the
transverse volume $V_{tr}$ of the foliation. The results of Section~\ref%
{special} give more specific information in special cases. For example, if
the leaf closure codimension is one, then the spectrum of the basic
Laplacian determines the $L^{2}$ norm of the mean curvature of the leaf
closure foliation. Therefore, the spectrum determines whether or not the
leaf closure foliation is minimal.

In most cases considered in the paper, we assume that the foliations are
transversally oriented for simplicity. In Section~\ref{nonorient}, we
describe the method of obtaining the asymptotics of the basic heat kernel on
Riemannian foliations that are not transversally orientable.

\section{Heat Kernels and Operators on the Basic Manifold}

\label{setup}In this section, we introduce some notation, recall some
results contained in \cite{Ri2}, and then use these results to obtain a
formula for the trace of the basic heat kernel. Let $M$ be an $n$%
-dimensional, closed, connected, oriented Riemannian manifold without
boundary, and let $\mathcal{F}$ be a transversally--oriented, codimension $q$
foliation on $M$ for which the metric is bundle--like. As in the
introduction, we let $\Delta _{B}$ denote the basic Laplacian, and we let $%
K_{B}(t,x,y)$ be the basic heat kernel on functions defined in~(\ref%
{basicdef}).

Let $\widehat M$ be the oriented transverse orthonormal frame bundle of $(M,%
\mathcal{F})$, and let $\pi$ be the natural projection $\pi:\widehat
M\longrightarrow M$. The manifold $\widehat M$ is a principal $SO(q)$-bundle
over $M$. Given $\hat x\in \widehat M$, let $\hat xg$ denote the
well-defined right action of $g\in SO(q)$ applied to $\hat x$. Associated to 
$\mathcal{F}$ is the lifted foliation $\widehat{\mathcal{F}}$ on $\widehat M$
. The lifted foliation is transversally parallelizable, and the closures of
the leaves are fibers of a fiber bundle $\rho:\widehat M\longrightarrow W$.
The manifold $W$ is smooth and is called the basic manifold (see \cite[%
pp.~105-108, p.~147ff]{Mo}). Let $\overline{\mathcal{F}}$ denote the
foliation of $\widehat M$ by leaf closures of $\widehat {\mathcal{F}}$.

Endow $\widehat M$ with the metric $g^M+g^{SO(q)}$, where $g^M$ is the
pullback of the metric on $M$, and $g^{SO(q)}$ is the standard, normalized,
biinvariant metric on the fibers. By this, we mean that we use the
transverse Levi--Civita connection (see \cite[p.~80ff]{Mo}) to do the
following. We calculate the inner product of two horizontal vectors in $%
T_{\hat x}\widehat M$ by using $g^M$, and we calculate the inner product of
two vertical vectors using $g^{SO(q)}$. We require that vertical vectors are
orthogonal to horizontal vectors. This metric is bundle--like for both $%
(\widehat M, \widehat {\mathcal{F}})$ and $(\widehat M,\overline {\mathcal{F}%
})$. The transverse metric on $(\widehat M,\overline {\mathcal{F}})$ induces
a well--defined Riemannian metric on $W$. The group $G=SO(q)$ acts by
isometries on $W$ according to $\rho(\hat x)g:=\rho(\hat xg)$ for $g\in
SO(q) $.

The volume form on $\widehat{M}$ can be written as $\mathrm{dvol}_{\overline{%
\mathcal{F}}}\,\rho ^{\ast }\!\mathrm{dvol}_{W}$, where $\mathrm{dvol}_{%
\overline{\mathcal{F}}}$ is the volume form of any leaf closure and $\mathrm{%
dvol}_{W}$ is the volume form on the basic manifold $W$. Let $\phi
:W\rightarrow \mathbb{R}$ be defined by taking $\phi (y)$ to be the volume
of $\rho ^{-1}(y)$. The function $\phi $ is obviously positive and is also
smooth, since $\rho $ is a smooth Riemannian submersion (see a proof of a
similar fact in \cite[Proposition 1.1]{PaRi}). Let $(\ ,\ )$ denote the
pointwise inner product of forms, and let $\left\langle \ ,\ \right\rangle $
denote the $L^{2}$-inner product of forms. Then for all $\alpha ,\gamma \in
\Omega ^{\ast }(W)$, 
\begin{eqnarray}
\left\langle \rho ^{\ast }\alpha ,\rho ^{\ast }\gamma \right\rangle _{%
\widehat{M}} &=&\int_{\widehat{M}}\left( \rho ^{\ast }\alpha ,\rho ^{\ast
}\gamma \right) _{\widehat{M}}\,\mathrm{dvol}_{\overline{\mathcal{F}}}\,\rho
^{\ast }\!\mathrm{dvol}_{W}  \notag \\
&=&\int_{\widehat{M}}\rho ^{\ast }\left( \alpha ,\gamma \right) _{W}\,%
\mathrm{dvol}_{\overline{\mathcal{F}}}\,\rho ^{\ast }\!\mathrm{dvol}_{W} 
\notag \\
&=&\int_{W}\phi \cdot \left( \alpha ,\gamma \right) _{W}\,\mathrm{dvol}%
_{W}=\left\langle \phi \alpha ,\gamma \right\rangle _{W}.  \label{metricw}
\end{eqnarray}
We have used the fact that for a bundle--like metric, the pointwise inner
product has the same action on basic forms as the pullback of the pointwise
inner product on the local quotient manifold. For our foliation $(\widehat{M}%
,\overline{\mathcal{F}})$, $W$ is the (global) quotient manifold.

Let $\widehat \Delta_B$ denote the basic Laplacian associated to the lifted
foliation on $\widehat M$, and let $\Delta_{W}$ denote the ordinary
Laplacian on $W$ corresponding to the induced metric on $W$. Note that $\phi$
is invariant under the right action of $SO(q)$ on $W$, so we define the
smooth function $\psi:M\to \mathbb{R}$ by $\pi^{*}\psi=\rho^{*}\phi$. Let $%
\sigma \lrcorner $ denote the adjoint of the wedge product $\sigma \wedge$
for any form $\sigma$.

We define the elliptic operator $\widetilde{\Delta_{W}}:C^\infty (W)\to%
\mathbb{R}$ by 
\begin{eqnarray}
\widetilde{\Delta_{W}}&=&\Delta_{W}-\frac 1\phi\, (d\phi )\lrcorner \circ d
\label{twistlap} \\
&=&-g^{ij}\partial_{i}\partial_{j}-b^{j}\partial_{j}\text{ locally on }W, 
\notag
\end{eqnarray}
where $g=\left( g_{ij}\right)$ is the metric on $W$ in local coordinates, $%
\left( g^{ij}\right)=g^{-1}$, and $b^{j}=\partial_{i}g^{ij}+g^{ij}%
\partial_{i}\left( \log \left( \phi\sqrt{\det g}\right)\right)$. Let $K_{B}$
and $\widehat {K_{B}}$ denote the basic heat kernels on $M$ and $\widehat M$%
, respectively, and let $\widetilde {K_{W}}$ denote the heat kernel
corresponding to $\widetilde{\Delta_{W}}$ on $W$. Then we have the following
results (see \cite[Theorems 1.1 and 2.4]{Ri2}):

\begin{enumerate}
\item The following equation holds on $C^{\infty }(W)$: 
\begin{equation*}
\widehat{\Delta }_{B}\rho ^{*}=\rho ^{*}\widetilde{\Delta _{W}}.
\end{equation*}

\item For every $x,y\in M$, $\hat{x}\in \pi ^{-1}(x)$, and $\hat{y}\in \pi
^{-1}(y)$, 
\begin{equation*}
K_{B}(t,x,y)=\int_{G}\widetilde{K_{B}}(t,\hat{x},\hat{y}g)\,\chi (g).
\end{equation*}

\item For every $\hat{x}\in \pi ^{-1}(x)$ and $\hat{y}\in \pi ^{-1}(y)$, 
\begin{equation}
\widehat{K_{B}}(t,\hat{x},\hat{y})=\frac{\widetilde{K_{W}}\left( t,\rho (%
\hat{x}),\rho (\hat{y})\right) }{\phi \left( \rho (\hat{y})\right) }.
\label{kerhatkerw}
\end{equation}
\end{enumerate}

We will now write the trace of the basic heat kernel in terms of $\widetilde{%
K_{W}}$. Let $\mathrm{dvol}$, $\widehat{\mathrm{dvol}}$, $\mathrm{dvol}_{W}$%
, and $\chi $ denote the volume forms on $M$, $\widehat{M}$, $W$, and $%
G=SO(q)$, respectively. Then 
\begin{eqnarray}
K_{B}(t) &=&\text{trace}\,e^{-t\Delta _{B}}  \notag \\
&=&\int_{M}K_{B}(t,x,x)\,\mathrm{dvol}  \notag \\
&=&\int_{M}\int_{G}\widehat{K_{B}}\left( t,\hat{x},\hat{x}g\right) \,\chi
(g)\,\mathrm{dvol}(x).  \label{trmhat}
\end{eqnarray}%
For any measurable section $s:M\rightarrow \widehat{M}$ that is smooth on an
open, dense subset of $M$, we can describe points of $\widehat{M}$ in terms
of the map $M\times G\rightarrow \widehat{M}$ defined by $(x,h)\mapsto s(x)h$%
. In these \textquotedblleft coordinates,\textquotedblright\ the measure on $%
\widehat{M}$ is $\widehat{\mathrm{dvol}}(x,h)=\mathrm{dvol}(x)\,\chi (h)$.
Therefore, if we change coordinates $\hat{x}\mapsto \hat{x}h$ in~(\ref%
{trmhat}), average over $G$, and use Fubini's Theorem, we get 
\begin{eqnarray*}
K_{B}(t) &=&\int_{G}\int_{G}\int_{M}\widehat{K_{B}}\left( t,\hat{x}h,\hat{x}%
hg\right) \,\mathrm{dvol}(x)\,\chi (h)\,\chi (g) \\
&=&\int_{G}\left\langle \widehat{K_{B}}(t,\cdot ,\cdot \,g),1\right\rangle _{%
\widehat{M}}\,\chi (g) \\
&=&\int_{G}\left\langle \frac{\widetilde{K_{W}}(t,\rho (\cdot ),\rho (\cdot
)g)}{\phi (\rho (\cdot ))},1\right\rangle _{\widehat{M}}\,\chi (g)\,\,\,%
\text{by (\ref{kerhatkerw})} \\
&=&\int_{G}\left\langle \widetilde{K_{W}}(t,\cdot ,\cdot \,g),1\right\rangle
_{W}\,\chi (g)\,\,\,\,\,\,\,\,\text{by (\ref{metricw})} \\
&=&\int_{G\times W}\widetilde{K_{W}}(t,w,wg)\,\mathrm{dvol}_{W}(w)\,\chi (g)
\end{eqnarray*}%
We have shown the following:

\begin{proposition}
\label{trbasic1} The trace $K_{B}(t)$ of the basic heat kernel on functions
is given by the formula 
\begin{equation*}
K_{B}(t)=\int_{G\times W}\widetilde{K_{W}}(t,w,wg)\,\mathrm{dvol}%
_{W}(w)\,\chi (g).
\end{equation*}
\end{proposition}

\begin{corollary}
The trace of the basic heat kernel on functions on $M$ is the same as the
trace of the heat kernel corresponding to $\widetilde{\Delta _{W}}$
restricted to $SO(q)$--invariant functions on $W$.
\end{corollary}

Therefore, the results of \cite{BrH2} apply, since $\widetilde {\Delta_{W}}$
has the same principal symbol as $\Delta_{W}$ and commutes with the $G$%
--action. In particular, formula~(\ref{equtrace}) holds, where $m=\dim W/G$, 
$K_{0}$ is less than or equal to the number of different dimensions of $G$%
-orbits in $W$, and $a_{0}=$Vol$(W/G)$.

The leaf closures of $(M,\mathcal{F})$ with maximal dimension form an open,
dense subset $M_{0}$ of $M$ (see \cite[pp. 157--159]{Mo}). The leaf closures
of the lifted foliation cover the leaf closures of $M$ (see \cite[p.~151ff]%
{Mo}). Given a leaf closure $\overline{L_{x}}$ containing $x\in M$, the
dimension of a leaf closure contained in $\pi^{-1}\left( \overline{L_{x}}%
\right)$ is $\dim \overline{L_{x}}+\dim H_{\hat x}$, where $\hat x\in
\pi^{-1}\left( x\right)$ and $H_{\hat x}$ is the subgroup of $SO(q)$ that
fixes the leaf closure containing $\hat x$. In some sense, the group $%
H_{\hat x}$ measures the holonomy of the leaves contained in the leaf
closure containing $x$ as well as the holonomy of the leaf closure. The
group $H_{\hat x}$ is isomorphic to the structure group corresponding to the
principal bundle $\overline L_{\hat x}\to \overline L_{x}$, where $\overline
L_{\hat x}$ is the leaf closure in $\widehat M$ that contains $\hat x$; the
conjugacy class of $H\subset G$ depends only on the leaf $L_{x}$. Therefore,
on an open, dense subset of $W$, the orbits of $G$ have dimension $\frac {%
q(q-1)}2-m$, where $m$ is the dimension of the principal isotropy groups.
The above discussion also implies that the dimension of $W$ is $\overline q+%
\frac {q(q-1)}2-m$, where $\overline q$ is the codimension of the leaf
closures of $M$ of maximal dimension. As a result, we have that 
\begin{equation*}
m=\dim W/G=\dim W-\left(\frac {q(q-1)}2-m\right) =\overline q.
\end{equation*}
Also, the number of different dimensions of $G$-orbits in $W$ is equal to
the number of different dimensions of leaf closures in $(M,\mathcal{F})$.

For $w\in W$, let $wG$ denote the $G$-orbit of $w$. By construction, 
\begin{eqnarray*}
\text{Vol}(wG)\cdot \phi (w) &=&\text{Vol}\left( \rho ^{-1}(wG)\right) \\
&=&\text{Vol}\left( \overline{L_{x}}\right) ,
\end{eqnarray*}
where $x$ is chosen so that $x\in \pi \left( \rho ^{-1}(wG)\right) $. Using
the above information, we have 
\begin{eqnarray*}
\text{Vol}(W/G) &=&\int_{W}\frac{1}{\text{Vol}(wG)}\,\mathrm{dvol}_{W}(w) \\
&=&\int_{W}\frac{\phi (w)}{\text{Vol}\left( \rho ^{-1}(wG)\right) }\,\mathrm{%
dvol}_{W}(w) \\
&=&\int_{\widehat{M}}\frac{1}{\text{Vol}\left( \rho ^{-1}(\rho (\hat{x}%
)G)\right) }\,\widehat{\mathrm{dvol}}(\hat{x})\qquad \text{by (\ref{metricw})%
} \\
&=&\int_{\widehat{M}}\frac{1}{\text{Vol}\left( \overline{L_{\pi (\hat{x})}}%
\right) }\,\widehat{\mathrm{dvol}}(\hat{x}) \\
&=&\int_{M}\frac{1}{\text{Vol}\left( \overline{L_{x}}\right) }\,{\mathrm{dvol%
}}(x)
\end{eqnarray*}
We remark that the integrals above converge. Proposition 2.1 in \cite{Pa},
which concerns isometric flows, is easily modified to show that the first
integral above converges; the convergence of the other integrals follows.

Using the above discussions and the results of \cite{BrH2}, we have the
following:

\begin{theorem}
\label{tracebasic} Let $\overline{q}$ be the codimension of the leaf
closures of $(M,\mathcal{F})$ with maximal dimension. As $t\rightarrow 0$,
the trace $K_{B}(t)$ of the basic heat kernel on functions satisfies the
following asymptotic expansion for any positive integer $J$: 
\begin{equation*}
K_{B}(t)=\frac{1}{(4\pi t)^{\overline{q}/2}}\left( a_{0}+\sum_{j>0,~0\leq
k\leq K_{0}}a_{jk}t^{j/2}(\log t)^{k}+O\left( t^{\frac{J+1}{2}}(\log
t)^{K_{0}-1}\right) \right) ,
\end{equation*}%
where $K_{0}$ is less than or equal to the number of different dimensions of
leaf closures of $(M,\mathcal{F})$, and where 
\begin{equation*}
a_{0}=V_{tr}=\int_{M}\frac{1}{\mathrm{Vol}\left( \overline{L_{x}}\right) }\,{%
\mathrm{dvol}}(x).
\end{equation*}%
The coefficients of the logarithmic terms in the expansion vanish if the
codimension of the foliation is less than four.
\end{theorem}

In the last statement of the above proposition, we used the fact that $SO(q)$
is connected, has rank one for $q<4$, and acts effectively on $W$. Also,
note that $a_{jk}$ depends only on the transverse metric, because by the
results in \cite{BrH2} it depends only on infinitesimal metric information
on the set $\{ (g,w)\, |\, wg=w\}$, which is entirely determined by the
transverse geometry of $(M,\mathcal{F})$.

We will call the coefficients $a_{0}$ and $a_{jk}$ the basic heat
invariants; they are functions of the spectrum of the basic Laplacian,
because of the formula (see \cite[Theorem 3.5]{PaRi}) 
\begin{equation*}
K_{B}(t)=\text{trace}\left( e^{-t\Delta_B}\right) =\sum_{j\geq
0}e^{-t\lambda^{B}_{j}}.
\end{equation*}
Using Karamata's theorem (\cite[pp.~418--423]{Fel}), we also have the
following, which generalizes Weyl's asymptotic formula (\cite{We}):

\begin{corollary}
\label{Weyl} Let $0=\lambda _{0}<\lambda _{1}^{B}\leq \lambda _{2}^{B}\leq
\ldots $ be the eigenvalues of the basic Laplacian on functions, counting
multiplicities. Then the spectral counting function $N_{B}(\lambda )$
satisfies the following asymptotic formula: 
\begin{eqnarray*}
N_{B}(\lambda ) &:&=\#\{\lambda _{m}^{B}|\lambda _{m}^{B}\leq \lambda \} \\
&\sim &\frac{V_{\text{tr}}}{(4\pi )^{\overline{q}/2}\Gamma \left( \frac{%
\overline{q}}{2}+1\right) }\lambda ^{\overline{q}/2}
\end{eqnarray*}
as $\lambda \rightarrow \infty $, where $\overline{q}$ and $V_{\text{tr}}$
are defined as above.
\end{corollary}

Observe that we are able to prove Theorem~\ref{tracebasic} and Corollary~\ref%
{Weyl} with very little information about the heat kernel $\widetilde{K_{W}}$%
; we used Proposition~\ref{trbasic1} and the results in \cite{BrH2} alone.
We also remark that although the expansion~(\ref{equtrace}) contains
logarithmic terms, no examples for which these terms are nonzero are known.
In the proof of this expansion \cite{BrH2}, the authors show that $G\times M$
can be decomposed into pieces over which the integral has an expansion with
possibly nonzero logarithmic terms. In the cases for which the authors
proved the nonexistence of logarithmic terms, symmetries cause the sum of
these logarithmic contributions to vanish. We make the following conjectures:

\begin{conjecture}
\label{Gconjecture} Suppose that $\Gamma $ is a compact group that acts
isometrically and effectively on a compact, connected Riemannian manifold $M$%
. Then the coefficients $a_{jk}$ of the equivariant trace formula (\ref%
{equtrace}) satisfy the following:

\begin{itemize}
\item $a_{jk}=0$ for $k>0$.

\item If $M$ is oriented and $\Gamma $ acts by orientation--preserving
isometries, then $a_{j0}=0$ for $j$ odd.
\end{itemize}
\end{conjecture}

\begin{conjecture}
\label{traceconjecture} (Corollary of Conjecture~\ref{Gconjecture}) In
Theorem~\ref{tracebasic}, $a_{jk}=0$ for $k>0$. If in addition $SO\left(
q\right) $ acts by orientation preserving isometries on $W$, then $a_{j0}=0$
for $j$ odd.
\end{conjecture}

We remark that since our foliation is transversally oriented, $SO\left(
q\right) $ acts by orientation preserving isometries precisely when the leaf
closures are all transversally orientable. The $SO\left( q\right) $ action
is not always orientation preserving, as Example~\ref{exnonorientable} shows.

\section{Formulas for the Basic Heat Invariants in Special Cases}

\label{special}We will now explicitly derive the asymptotics of the integral
in Proposition~\ref{trbasic1} in some special cases.

\subsection{Regular Closure\label{finite}}

Suppose that $(M,\mathcal{F})$ has \textit{regular closure}. In other words,
assume that the leaf closures all have the same dimension. Note that this
implies that all of the leaves and leaf closures have finite holonomy. In
this case, the orbits of $SO\left( q\right) $ or $O\left( q\right) $ (and
thus the leaf closures of $(M,\mathcal{F})$) all have the same dimension,
and the space of leaf closures is a Riemannian orbifold. Since the orbits
all have the same dimension, locally defined functions of the metric along
the orbits are smooth (bounded) functions on the basic manifold, and the
volumes of the orbits (and hence the volumes of the leaf closures) are
bounded away from zero. Therefore, the coefficients in the asymptotic
expansion for $K_{B}(t,x,x)$ found in \cite{Ri2} are bounded on the foliated
manifold $M$, but the error term is not necessarily bounded. These local
expressions cannot in general be integrated over $M$ to yield the
asymptotics of the trace of the basic heat operator, as in the method used
to obtain the asymptotic expansion of the trace of the ordinary heat
operator on a manifold, as described in the introduction. Instead, a
calculation of the trace of a second order operator on an orbifold is
required.

Note that the leaf closures of $(M,\mathcal{F})$ themselves form a
Riemannian foliation $(M,\mathcal{F}^{c})$, in which all leaves are compact.
In general the basic Laplacian $\Delta _{B}$ on functions satisfies $\Delta
_{B}=P_{B}\delta d$, where $d$ is the exterior derivative, $\delta $ is the $%
L^{2}$ adjoint of $d$, and $P_{B}$ is the orthogonal projection of $L^{2}$
functions to the $L^{2}$ of basic functions (see \cite{PaRi}). Since the
projection $P_{B}$ on functions is identical for both foliations $(M,%
\mathcal{F})$ and $(M,\mathcal{F}^{c})$, the basic Laplacian on basic
functions of the foliation $(M,\mathcal{F})$ is the same as the basic
Laplacian $\Delta _{B}^{c}$ on basic functions of the foliation $(M,\mathcal{%
F}^{c})$ . Note that the equivalent statement for the basic Laplacian on
forms is false.

Thus, it suffices to solve the problem of calculating the basic heat kernel
asymptotics for the case of closed leaves. Let $p:M\rightarrow N=M\slash 
\mathcal{F}^{c}$ be the quotient map, which is a Riemannian submersion away
from the leaf closures with holonomy. Similar to the arguments in Section %
\ref{setup}, the basic Laplacian on functions satisfies 
\begin{equation*}
\Delta _{B}\circ p^{\ast }=p^{\ast }\circ \left( \Delta ^{N}-\frac{d\psi }{%
\psi }\lrcorner \circ d\right) ,
\end{equation*}
where $\Delta ^{N}=\delta d$ is the ordinary Laplacian on the orbifold $N$,
and $\psi $ is the function on $N$ defined by $\psi \left( x\right) =\mathrm{%
Vol}\left( p^{-1}\left( x\right) \right) $ if $p^{-1}\left( x\right) $ is a
principal leaf closure and extended to be continuous (and smooth) on $N$.
Note that if $\kappa ^{c}$ is the mean curvature form of $(M,\mathcal{F}%
^{c}) $ and $P_{B}^{1}$ is the $L^{2}$ projection from one-forms to basic
one-forms, then $P_{B}^{1}\kappa ^{c}=-p^{\ast }\left( \frac{d\psi }{\psi }%
\right) $. Since $f$ is basic function on $M$ if and only if $f=p^{\ast }g$
for some function $g$ on $N$, the trace of $K_{B}(t)$ of the basic heat
kernel on functions on $M$ is the trace of $e^{-t\left( \Delta ^{N}-\frac{%
d\psi }{\psi }\lrcorner \circ d\right) }$ on functions on the orbifold $N$.
To calculate this trace, we first collect the following known results.

\begin{lemma}
(See \cite{Cha}, \cite{Gr}, \cite{Ro}) Let $L$ be a second order operator on
functions on a closed Riemannian manifold $N$ of dimension $m$, such that $%
L=\Delta +V+Z,$ where $\Delta $ is the Laplace operator, $V$ is a purely
first-order operator, and $Z$ is a zeroth order operator. Then, the heat
kernel $K_{L}\left( t,x,y\right) $ of $L$, the fundamental solution of the
operator $\frac{\partial }{\partial t}+L$, exists and has the following
properties:

\begin{enumerate}
\item Given $\varepsilon >0$, there exists $c>0$ such that if $r\left(
x,y\right) =\mathrm{dist}\left( x,y\right) >\varepsilon $, then $K_{L}\left(
t,x,y\right) =\mathcal{o}\left( e^{-c/t}\right) $ as $t\rightarrow 0$.

\item If $r\left( x,y\right) =\mathrm{dist}\left( x,y\right) $ is
sufficiently small, then as $t\rightarrow 0$, 
\begin{equation*}
K_{L}\left( t,x,y\right) =\frac{e^{-r^{2}\left( x,y\right) /4t}}{\left( 4\pi
t\right) ^{m/2}}\left( c_{0}\left( x,y\right) +c_{1}\left( x,y\right)
t+...+c_{k}\left( x,y\right) +\mathcal{O}\left( t^{k+1}\right) \right)
\end{equation*}
for any $k$, each $c_{j}\left( x,y\right) $ is smooth, and $c_{0}\left(
x,x\right) =1$. The function $c_{j}\left( x,y\right) $ is determined by the
metric and the operators $V$ and $Z$ and their derivatives, evaluated along
the minimal geodesic connecting $x$ and $y$.
\end{enumerate}
\end{lemma}

If the manifold $N$ is instead a Riemannian orbifold, the operators still
can be defined (by their definitions on the local covers using pullbacks),
and a fundamental solution to the heat equation still exists. Note that in
all such cases, the lifted operator $\widetilde{L}=\widetilde{\Delta }+%
\widetilde{V}+\widetilde{Z}$ is equivariant with respect to the local finite
group action. Since the asymptotics of the heat kernel are still determined
locally, we have the following corollary.

\begin{lemma}
Let $L$ be a second order operator on functions on a closed Riemannian
orbifold $N$ of dimension $m$, such that $L=\Delta +V+Z$, where $\Delta $ is
the Laplace operator, $V$ is a purely first order operator, and $Z$ is a
zeroth order operator. Then, the heat kernel $K_{L}\left( t,x,y\right) $ of $%
L$, the fundamental solution of the operator $\frac{\partial }{\partial t}+L$%
, exists and has the following properties:

\begin{enumerate}
\item Given $\varepsilon >0$, there exists $c>0$ such that if $r\left(
x,y\right) =\mathrm{dist}\left( x,y\right) >\varepsilon $, then $K_{L}\left(
t,x,y\right) =\mathcal{O}\left( e^{-c/t}\right) $ as $t\rightarrow 0$.

\item If $r\left( x,y\right) =\mathrm{dist}\left( x,y\right) $ is
sufficiently small, then as $t\rightarrow 0$, and if the minimal geodesic
connecting $x$ and $y$ is away from the singular set of the orbifold, then 
\begin{equation}
K_{L}\left( t,x,y\right) =\frac{e^{-r^{2}\left( x,y\right) /4t}}{\left( 4\pi
t\right) ^{m/2}}\left( c_{0}\left( x,y\right) +c_{1}\left( x,y\right)
t+...+c_{k}\left( x,y\right) +\mathcal{O}\left( t^{k+1}\right) \right)
\label{localKernelOrbifold}
\end{equation}
for any $k$, each $c_{j}\left( x,y\right) $ is smooth, and $c_{0}\left(
x,x\right) =1$.

\item Let $z$ be an element of the singular set of $N$, and let $H_{z}$
denote the finite group of isometries such that a neighborhood $U$ of $z$ in 
$N$ is isometric to $\widetilde{U}\slash H_{z}$, where $\widetilde{U}$ is
an open set in $\mathbb{R}^{m}$ with the given metric. Let $o\left(
H_{z}\right) $ denote the order of $H_{z}$. Assuming that the neighborhood $U
$ is sufficiently small, there exists $c>0$ such that if $x,y\in U$, then 
\begin{equation*}
K_{L}\left( t,x,y\right) =\frac{1}{o\left( H_{z}\right) }\sum_{h\in H_{z}}%
\widetilde{K_{L}}\left( t,x,hy\right) +\mathcal{O}\left( e^{-c/t}\right) ,
\end{equation*}%
where $\widetilde{K_{L}}$ is the heat kernel of the lifted operator $%
\widetilde{\Delta }+\widetilde{V}+\widetilde{Z}$ on $\widetilde{U}$, which
itself satisfies an asymptotic expansion as in (\ref{localKernelOrbifold})
above.
\end{enumerate}
\end{lemma}

The results above are well-known and well-utilized in the cases where $V=Z=0$
(see, for example, \cite{D1}, \cite{D2}), but they are true in the
generality stated.

Next, we establish an estimate and a trigonometric identity.

\begin{lemma}
\label{integralAsymptGamma}Given $a>0$ and $b\in \mathbb{N}$, we have 
\begin{eqnarray*}
\int_{0}^{\varepsilon }e^{-\frac{x^{2}a^{2}}{t}}x^{b}dx &=&\int_{0}^{\infty
}e^{-\frac{x^{2}a^{2}}{t}}x^{b}dx-\int_{\varepsilon }^{\infty }e^{-\frac{%
x^{2}a^{2}}{t}}x^{b}dx \\
&=&\frac{\Gamma \left( \frac{b+1}{2}\right) }{2a^{b+1}}t^{\frac{b+1}{2}}+%
\mathcal{O}\left( \left( \frac{\varepsilon ^{2}a^{2}}{t}\right) ^{\frac{b-1}{%
2}}e^{-\frac{\varepsilon ^{2}a^{2}}{t}}\right) ~ \\
&=&\frac{\Gamma \left( \frac{b+1}{2}\right) }{2a^{b+1}}t^{\frac{b+1}{2}}+%
\mathcal{O}\left( \left( \frac{t}{\varepsilon ^{2}a^{2}}\right) ^{N}\right) ~%
\text{as }t\rightarrow 0,
\end{eqnarray*}
for any $N\geq \frac{b+1}{2}$.
\end{lemma}

\begin{proof}
Substituting $u=\frac{x^{2}a^{2}}{t}$, or $x=$ $\frac{\sqrt{ut}}{a}$, we get%
\begin{eqnarray*}
\int_{0}^{\varepsilon }e^{-\frac{x^{2}a^{2}}{t}}x^{b}dx &=&\frac{t^{\frac{b+1%
}{2}}}{2a^{b+1}}\int_{0}^{\varepsilon ^{2}a^{2}/t}e^{-u}u^{\frac{b-1}{2}}du
\\
&=&\frac{t^{\frac{b+1}{2}}}{2a^{b+1}}\int_{0}^{\infty }e^{-u}u^{\frac{b-1}{2}%
}du-\frac{t^{\frac{b+1}{2}}}{2a^{b+1}}\int_{\varepsilon ^{2}a^{2}/t}^{\infty
}e^{-u}u^{\frac{b-1}{2}}du \\
&=&\frac{t^{\frac{b+1}{2}}\Gamma \left( \frac{b+1}{2}\right) }{2a^{b+1}}-%
\frac{t^{\frac{b+1}{2}}}{2a^{b+1}}\Gamma \left( \frac{b+1}{2},\frac{%
\varepsilon ^{2}a^{2}}{t}\right) ,
\end{eqnarray*}%
where $\Gamma \left( A,z\right) $ is the (upper) incomplete Gamma function.
It is known that $\Gamma \left( A,z\right) $ is proportional to $%
e^{-z}z^{A-1}\left( 1+\mathcal{O}\left( \frac{1}{z}\right) \right) $ as $%
\left\vert z\right\vert \rightarrow \infty $, and the formulas above follow.
\end{proof}

\begin{lemma}
\label{trigLemma}For any positive integer $k$,%
\begin{equation*}
\sum_{j=1}^{k-1}\frac{1}{\sin ^{2}\left( \frac{\pi j}{k}\right) }=\frac{%
k^{2}-1}{3}.
\end{equation*}
\end{lemma}

\begin{proof}
Many thanks to George Gilbert. Contact the author for a proof.
\end{proof}

\medskip The goal is evaluate the asymptotics of 
\begin{equation*}
K_{B}\left( t\right) =\mathrm{tr}\left( \left. e^{-tP}\right\vert
_{C^{\infty }\left( N\right) }\right) =\int_{N}K_{P}\left( t,x,x\right) ~%
\mathrm{dvol}
\end{equation*}%
as $t\rightarrow 0$, where $P=\Delta ^{N}-\frac{d\psi }{\psi }\lrcorner
\circ d$, but we first proceed with calculating the heat trace of a general
operator $L$ as in (\ref{localKernelOrbifold}) on an orbifold. We now
decompose $N$ as follows. Given an element $z\in N$, let $H_{z}$ denote a
subgroup of the orthogonal group $O\left( \dim N\right) $ such that every
sufficiently small metric ball around $z$ is isometric to a ball in $%
H_{z}\backslash \left( \mathbb{R}^{\dim N},g\right) $, where $g$ is an $H_{z}$%
-invariant metric. The conjugacy class $\left[ H_{z}\right] $ in $O\left(
\dim N\right) $ is called isotropy type of $z$. The stratification of $N$ is
a partition of $N$ into the different isotropy types. The partial order on
these isotropy types is defined as in the general $G$-manifold structure
(see Section \ref{general}). Let $o\left( H_{z}\right) $ denote the order of 
$H_{z}$.

As in Section \ref{general}, we decompose $N$ into pieces which include
tubular neighborhoods of parts of the singular strata of the orbifold and
the principal stratum (for which $H_{z}=\left\{ e\right\} $) minus the other
neighborhoods. We may further decompose $N=N_{\varepsilon }\cup \coprod
N_{j}$ as a finite disjoint union, where each $N_{j}\ $is of the form $%
H_{j}\backslash \widetilde{N_{j}}$ with $\widetilde{N_{j}}$ contractible, no
nontrivial element of $H_{j}$ fixing all of $\widetilde{N_{j}}$, and with at
least one point of $\widetilde{N_{j}}$ having isotropy $H_{j}$. We may think
of $N_{j}$ as a tubular neighborhood of an open subset of a singular
stratum, up to sets of measure zero.

Given any isometry $h\in H_{i}\setminus \left\{ e\right\} $, choose a
tubular neighborhood 
\[
U_{h,\varepsilon }^{i}=T_{\varepsilon }\left(
H_{j}\backslash \left( \widetilde{N_{i}}\right) ^{h}\right) \cap N_{i}
\]
 of
the local submanifold $S_{i}^{h}=\left( H_{j}\backslash \widetilde{N_{i}}%
\right) ^{h}$ of singular points fixed by $h$, and let $\widetilde{%
U_{h,\varepsilon }^{i}}=T_{\varepsilon }\left( \left( \widetilde{N_{i}}%
\right) ^{h}\right) \cap \widetilde{N_{i}}$ denote the local cover. Let $%
\widetilde{U_{e,\varepsilon }^{i}}$ $=\widetilde{N_{i}}$. 
\begin{equation*}
N_{\varepsilon }=N\setminus \bigcup\limits_{i}\left( \bigcup\limits_{h\in
H_{i}\setminus \left\{ e\right\} }U_{h,\varepsilon }^{i}\right) .
\end{equation*}%
Then there exists $c>0$ such that 
\begin{gather*}
K_{L}\left( t\right) +\mathcal{O}\left( e^{-c/t}\right) \\
=\int_{N_{\varepsilon }}K_{L}\left( t,x,x\right) ~\mathrm{dvol}+\sum_{i}%
\frac{1}{o\left( H_{i}\right) }\sum_{h\in H_{i}}\int_{\widetilde{%
U_{h,\varepsilon }^{i}}}\widetilde{K_{L}}\left( t,x,hx\right) ~\mathrm{dvol}
\\
=\frac{1}{\left( 4\pi t\right) ^{m/2}}\left( \int_{N_{\varepsilon }}%
\mathrm{dvol}+t\int_{N_{\varepsilon }}c_{1}\left( x,x\right) \,\mathrm{dvol}+%
\mathcal{O}\left( t^{2}\right) \right) \\
+\sum_{i}\frac{1}{o\left( H_{i}\right) }\left( \int_{\widetilde{N_{i}}}%
\mathrm{dvol}+t\int_{\widetilde{N_{i}}}c_{1}\left( x,x\right) \,\mathrm{dvol}%
+\mathcal{O}\left( t^{2}\right) \right) \\
+\sum_{i}\frac{1}{o\left( H_{i}\right) \left( 4\pi t\right) ^{m/2}}%
\sum_{h\in H_{i}\setminus \left\{ e\right\} }\int_{\widetilde{%
U_{h,\varepsilon }^{i}}}e^{-r^{2}\left( x,hx\right) /4t}\left( c_{0}\left(
x,hx\right) +\mathcal{O}\left( t^{1}\right) \right) ~\mathrm{dvol}
\end{gather*}%
Note that if $h\neq e$, $S_{i}^{h}$ is a disjoint union of connected
submanifolds $S_{i,j}^{h}$ codimension $d_{h}^{i,j}>0$. Then, following a
calculation in \cite{D1}, we may rewrite the last integral in geodesic
normal coordinates $x$. If $B_{\varepsilon }^{i,j}\left( y\right) $ denotes
the normal exponential ball of radius $\varepsilon $ at $y\in \widetilde{%
S_{i,j}^{h}}$, its volume form $\mathrm{dvol}_{B_{\varepsilon }^{i,j}}$
satisfies $\mathrm{dvol}_{B_{\varepsilon }^{i,j}}=\left( 1+\mathcal{O}\left(
\left\vert x\right\vert ^{2}\right) \right) ~dx$, and the volume form $%
\mathrm{dvol}$ on $\widetilde{U_{h,\varepsilon }^{i}}$ satisfies $\mathrm{%
dvol}=$ $\left( 1+\mathcal{O}\left( \left\vert x\right\vert ^{2}\right)
\right) ~\mathrm{dvol}_{B_{\varepsilon }^{i,j}}~\mathrm{dvol}_{\widetilde{%
S_{i,j}^{h}}}$ for each $j$. 
\begin{multline*}
\int_{\widetilde{U_{h,\varepsilon }^{i}}}e^{-r^{2}\left( x,hx\right) \slash
4t}\left( c_{0}\left( x,hx\right) +\mathcal{O}\left( t^{1}\right) \right) ~%
\mathrm{dvol} \\
=\sum_{j}\int_{\widetilde{S_{i,j}^{h}}}\int_{B_{\varepsilon
}^{i,j}}e^{-r^{2}\left( x,hx\right) \slash 4t}\left( 1+\mathcal{O}\left(
r^{2}\right) +\mathcal{O}\left( t\right) \right) ~\mathrm{dvol}%
_{B_{\varepsilon }^{i,j}}~\mathrm{dvol}_{\widetilde{S_{i,j}^{h}}} \\
=\sum_{j}\int_{\widetilde{S_{i,j}^{h}}}\int_{B_{\varepsilon
}^{i,j}}e^{-r^{2}\left( x,hx\right) \slash 4t}\left( 1+\mathcal{O}\left(
\left\vert x\right\vert ^{2}\right) +\mathcal{O}\left( t\right) \right) ~dx~%
\mathrm{dvol}_{\widetilde{S_{i,j}^{h}}} \\
=\sum_{j}\int_{\widetilde{S_{i,j}^{h}}}\int_{\left( I-h\right)
B_{\varepsilon }^{i,j}}e^{-r^{2}\left( u+hx\left( u\right) ,hx\left(
u\right) \right) \slash 4t}\left( 1+\mathcal{O}\left( \left\vert
u\right\vert ^{2}\right) +\mathcal{O}\left( t\right) \right) \cdot\\
\left\vert
\det \left( I-h\right) ^{-1}\right\vert du~\mathrm{dvol}_{\widetilde{%
S_{i,j}^{h}}},
\end{multline*}%
using the change of variables $u=\left( I-h\right) x$. Further (see \cite{D1}%
), there is a change of variables $y\left( u\right) $ such that $u_{j}=y_{j}+%
\mathcal{O}\left( \left\vert y\right\vert ^{3}\right) $ and $r^{2}\left(
u+hx\left( u\right) ,hx\left( u\right) \right) =\sum y_{j}^{2}$, and the
Jacobian for this change of variables is $1+\mathcal{O}\left( \left\vert
y\right\vert ^{2}\right) $. Thus, 
\begin{multline*}
\int_{\widetilde{U_{h,\varepsilon }^{i}}}e^{-r^{2}\left( x,hx\right)
\slash 4t}\left( c_{0}\left( x,hx\right) +\mathcal{O}\left( t^{1}\right)
\right) ~\mathrm{dvol} \\
=\left\vert \det \left( I-h\right) ^{-1}\right\vert \sum_{j}\int_{%
\widetilde{S_{i,j}^{h}}}\int_{y\left( \left( I-h\right) B_{\varepsilon
}^{i,j}\right) }e^{-\left\vert y\right\vert ^{2}\slash 4t}\cdot\\
\left( 1+\mathcal{%
O}\left( \left\vert y\right\vert ^{2}\right) +\mathcal{O}\left( t\right)
\right) ~dy~\mathrm{dvol}_{\widetilde{S_{i,j}^{h}}}.
\end{multline*}%
By Lemma \ref{integralAsymptGamma}, we have 
\begin{multline*}
\int_{y\left( \left( I-h\right) B_{\varepsilon }^{i,j}\right)
}e^{-\left\vert y\right\vert ^{2}\slash 4t}\left( 1+\mathcal{O}\left(
\left\vert y\right\vert ^{2}\right) +\mathcal{O}\left( t\right) \right) ~dy\\
=\int_{\mathbb{R}^{d_{h}^{i,j}}}e^{-\left\vert y\right\vert ^{2}\slash
4t}dy+\mathcal{O}\left( t^{\left( d_{z}+2\right) \slash 2}\right) \\
=\left( 2\cdot \frac{\Gamma \left( \frac{1}{2}\right) }{2\left( \frac{1}{2}%
\right) }\right) ^{d_{h}^{i,j}}t^{d_{h}^{i,j}\slash 2}+\mathcal{O}\left(
t^{\left( d_{h}^{i,j}+2\right) \slash 2}\right) \\
=\left( 4\pi t\right) ^{d_{h}^{i,j}\slash 2}+\mathcal{O}\left( t^{\left(
d_{h}^{i,j}+2\right) \slash 2}\right) .
\end{multline*}%
Thus, 
\begin{multline*}
\int_{\widetilde{U_{h,\varepsilon }^{i}}}e^{-r^{2}\left( x,hx\right) \slash
4t}\left( c_{0}\left( x,hx\right) +\mathcal{O}\left( t^{1}\right) \right) ~%
\mathrm{dvol}\\
=\left\vert \det \left( I-h\right) ^{-1}\right\vert \sum_{j}%
\mathrm{vol}\left( \widetilde{S_{i,j}^{h}}\right) \left( 4\pi t\right)
^{d_{h}^{i,j}\slash 2}+\mathcal{O}\left( t^{\left( d_{h}^{i,j}+2\right)
\slash 2}\right)
\end{multline*}

\vspace{1pt}Hence, letting $\varepsilon $ approach zero and summing up over
the neighborhoods of the singular strata of the orbifold $N$, we have 
\begin{gather}
K_{L}\left( t\right) =\frac{1}{\left( 4\pi t\right) ^{m/2}}\left( \int_{N}%
\mathrm{dvol}+t\int_{N}c_{1}\left( x,x\right) \,\mathrm{dvol}\right.\notag\\
+\mathcal{O}%
\left( t^{2}\right) \biggm) \notag \\
+\frac{1}{\left( 4\pi t\right) ^{m/2}}\sum_{i}\frac{1}{o\left(
H_{i}\right) }\sum_{h\in H_{i}\setminus \left\{ e\right\}
}\sum_{j}\left\vert \det \left( I-h\right) ^{-1}\right\vert \mathrm{vol}%
\left( \widetilde{S_{i,j}^{h}}\right) \left( 4\pi t\right)
^{d_{h}^{i,j}\slash 2}\notag\\
+\mathcal{O}\left( t^{\left( d_{h}^{i,j}+2\right)
\slash 2}\right)  \notag \\
=\frac{1}{\left( 4\pi t\right) ^{m/2}}\left( \mathrm{vol}\left( N\right)
+t\int_{N}c_{1}\left( x,x\right) \,\mathrm{dvol}\right)   \notag\\
+\frac{1}{\left( 4\pi t\right) ^{m/2}}\left( \sum_{i}\frac{1}{o\left(
H_{i}\right) }\sum_{h\in H_{i}\setminus \left\{ e\right\}
}\sum_{d_{h}^{i,j}=1,2}\left\vert \det \left( I-h\right) ^{-1}\right\vert 
\mathrm{vol}\left( \widetilde{S_{i,j}^{h}}\right) \left( 4\pi t\right)
^{d_{h}^{i,j}\slash 2}\right)  \notag\\
+\mathcal{O}\left( t^{\left( 3-m\right) \slash 2}\right) .
\label{asymptExpOrb1}
\end{gather}%
Note that $\widetilde{S_{i,j}^{h}}$ has codimension 1 precisely when $h$
acts as a reflection, in which case $\left\vert \det \left( I-h\right)
^{-1}\right\vert =\left( 1-\left( -1\right) \right) ^{-1}=\frac{1}{2}$.
Similarly, $\widetilde{S_{i,j}^{h}}$ has codimension 2 exactly when $h$ acts
as a rotation (say by $\theta _{h}=\frac{2\pi }{k}$ for some $k\in \mathbb{Z}%
_{>0}$) in the normal space to $\widetilde{S_{i,j}^{h}}$. In that case, 
\begin{multline*}
\left\vert \det \left( I-h\right) ^{-1}\right\vert =\left\vert \det \left( 
\begin{array}{cc}
1-\cos \theta _{h} & -\sin \theta _{h} \\ 
\sin \theta _{h} & 1-\cos \theta _{h}%
\end{array}%
\right) ^{-1}\right\vert\\
 =\frac{1}{2-2\cos \left( \theta _{h}\right) }=\frac{%
1}{4\sin ^{2}\left( \frac{\theta _{h}}{2}\right) }.
\end{multline*}%
If this number is summed over all nontrivial elements of a cyclic group
group of rotations by $\left\{ \frac{2\pi }{k},\frac{4\pi }{k},...,\frac{%
2\left( k-1\right) \pi }{k}\right\} $, by Lemma \ref{trigLemma} we have

\begin{equation*}
\sum_{j=1}^{k-1}\frac{1}{4\sin ^{2}\left( \frac{\pi j}{k}\right) }=\frac{%
k^{2}-1}{12}.
\end{equation*}%
In each of these cases, generic points $z$ of $S_{i,j}^{h}$ have isotropy
subgroups $H_{z}$ isomorphic to a cyclic group.

To obtain the asymptotic expansion, we identify two subsets of the singular
part of the orbifold:%
\begin{eqnarray*}
\Sigma _{\mathrm{ref}}N &=&\left\{ z\in N:H_{z}~\text{has order }2\text{ and
is generated by a reflection}\right\} , \\
\Sigma _{k}N &=&\big\{ z\in N:H_{z}\text{ is a cyclic group of order }k \\
&&\qquad \text{ and consists of rotations in a plane}\big\} .
\end{eqnarray*}%
Note that 
\begin{eqnarray*}
\mathrm{vol}^{m-1}\left( N_{i}\cap \Sigma _{\mathrm{ref}}N\right)  &=&\frac{2%
}{o\left( H_{i}\right) }\sum\limits_{h\text{ reflection}}\mathrm{vol}%
^{m-1}\left( \widetilde{S_{i}^{h}}\right)\\
 &=&\sum\limits_{h\text{ reflection}}%
\mathrm{vol}^{m-1}\left( S_{i}^{h}\right) , \\
\mathrm{vol}^{m-2}\left( N_{i}\cap \Sigma _{k}N\right)  &=&\frac{k}{o\left(
H_{i}\right) }\sum\limits_{h\text{ rotation of }\frac{2\pi }{k}}\mathrm{vol}%
^{m-2}\left( \widetilde{S_{i}^{h}}\right) \\
=\sum\limits_{h\text{ rotation of }%
\frac{2\pi }{k}}\mathrm{vol}^{m-2}\left( S_{i}^{h}\right)  
&=&\sum\limits_{h\text{ rotation by }\frac{2\pi j}{k}}\mathrm{vol}%
^{m-2}\left( S_{i}^{h}\right) \text{ for fixed }j\text{.}
\end{eqnarray*}%
We now combine the results above with (\ref{asymptExpOrb1}) to obtain the
following theorem.

\begin{theorem}\label{orbTheorem}
Let $L$ be a second order operator on functions on a closed Riemannian
orbifold $N$ of dimension $m$, such that $L=\Delta +V+Z$, where $\Delta $ is
the Laplace operator, $V$ is a purely first order operator, and $Z$ is a
zeroth order operator. Then, the trace of the heat operator has the
following asymptotic expansion as $t\rightarrow 0$:%
\begin{eqnarray*}
K_{L}\left( t\right)  &=&\frac{1}{\left( 4\pi t\right) ^{m/2}}\Biggm(\mathrm{%
vol}^{m}\left( N\right) +t^{1/2}\frac{\sqrt{\pi }}{2}\mathrm{vol}%
^{m-1}\left( \Sigma _{\mathrm{ref}}N\right)  \\
&&+t\left( \int_{N}c_{1}\left( x,x\right) \,\mathrm{dvol}+\frac{\pi \left(
k^{2}-1\right) }{3k}\mathrm{vol}^{m-2}\left( \Sigma _{k}N\right) \right) +%
\mathcal{O}\left( t^{3/2}\right) \Biggm),
\end{eqnarray*}%
where $c_{1}\left( x,x\right) $ is the heat trace coefficient from formula (%
\ref{localKernelOrbifold}) and $\Sigma _{\mathrm{ref}}N$ and $\Sigma _{k}N$
are parts of the singular stratum of the orbifold defined in the paragraph
above.
\end{theorem}

Note that the truth of this theorem is easily checked in the case of a
manifold with boundary or with a manifold quotient by a finite cyclic group
of rotations. Also, note that the coefficient $c_{1}\left( x,x\right) $ may
be computed using standard methods as in \cite{Ro}. 

We now wish to apply this result to the foliation case. Here, $N=M\slash 
\mathcal{F}^{c}$ is the leaf closure space (a Riemannian orbifold) of a
Riemannian foliation $\left( M,\mathcal{F}\right) $ with regular closure.
The operator of note is 
\begin{equation*}
K_{B}\left( t\right) =\mathrm{tr}\left( \left. e^{-tL}\right\vert
_{C^{\infty }\left( N\right) }\right) =\int_{N}K_{L}\left( t,x,x\right) ~%
\mathrm{dvol}
\end{equation*}%
as $t\rightarrow 0$, where $L=\Delta ^{N}-\frac{d\psi }{\psi }\lrcorner
\circ d=\Delta ^{N}+H^{c}$, and $H^{c}$ is the projection of the mean
curvature vector field of the foliation of $M$ by leaf closures to the set
of projectable vector fields, which implies that it descends to a vector
field on $N$. The formula needed is $c_{1}\left( x,x\right) $; we refer to 
\cite[formula (3.6)]{Ri2} for a similar calculation, which yields in our case%
\begin{eqnarray*}
c_{1}\left( x,x\right)  &=&\frac{S\left( x\right) }{6}+\frac{\Delta ^{N}\psi
\left( x\right) }{2\psi \left( x\right) }+\frac{1}{4}\left\vert H^{c}\left(
x\right) \right\vert ^{2} \\
&=&\frac{S\left( x\right) }{6}+\frac{\Delta ^{N}\psi \left( x\right) }{2\psi
\left( x\right) }+\frac{1}{4}\left\vert \frac{d\psi }{\psi }\right\vert
^{2}\left( x\right) ,
\end{eqnarray*}%
where $S\left( x\right) $ is the scalar curvature at $x\in N$. The theorem
below follows.

\begin{theorem}
Let $\left( M,\mathcal{F}\right) $ be a Riemannian foliation with regular
closure, so that the quotient orbifold $N=M\slash \mathcal{F}^{c}$ by leaf
closures has dimension $m$. If $x\in N$ corresponds to a principal leaf
closure, let $\psi \left( x\right) $ be the volume of the leaf closure, and
extend this function to be smooth on $N$. Let $S$ denote the scalar
curvature of $N$. Further, let $\Sigma _{\mathrm{ref}}N$ be the set of
singular points of $N$ corresponding to true boundary points, and let $%
\Sigma _{k}N$ be the set of singular points of $N$ which have neighborhoods
diffeomorphic to $\mathbb{R}^{n}$ mod a planar cyclic group of rotations of
order $k$. Then the trace of the basic heat kernel on functions satisfies
the following asymptotic formula as $t\rightarrow 0$. 
\begin{multline*}
K_{B}\left( t\right)  =\frac{1}{\left( 4\pi t\right) ^{m/2}}\Biggm(\mathrm{%
vol}^{m}\left( N\right) +t^{1/2}\frac{\sqrt{\pi }}{2}\mathrm{vol}%
^{m-1}\left( \Sigma _{\mathrm{ref}}N\right)  \\
+t\left( \int_{N}\frac{S\,}{6}+\frac{\Delta ^{N}\psi }{2\psi }+\frac{1}{4}%
\left\vert \frac{d\psi }{\psi }\right\vert ^{2}\mathrm{dvol}+\frac{\pi
\left( k^{2}-1\right) }{3k}\mathrm{vol}^{m-2}\left( \Sigma _{k}N\right)
\right) +\mathcal{O}\left( t^{3/2}\right) \Biggm).
\end{multline*}
\end{theorem}

\subsection{Transversally Oriented, Codimension One Riemannian Foliations}

\label{codim1}

Suppose that $\left( M,\mathcal{F}\right) $ is a transversally oriented,
codimension one Riemannian foliation. In this case, the analysis of the
basic manifold is unnecessary, because the basic manifold is isometric to
the space of leaf closures. For such a foliation, either the closure of
every leaf is all of $M$, or the leaves are all compact without holonomy. In
the first case, the basic Laplacian is identically zero, so that the trace
of the basic heat operator satisfies $K_{B}\left( t\right) =1$ for every $t$%
. In the case of compact leaves, the leaves are the fibers of a Riemannian
submersion over a circle. Thus, the basic functions are pullbacks of
functions on the circle, and the basic function $v:M\rightarrow \mathbb{R}$
given by $v(x)=($ the volume of the leaf containing $x)$ is smooth on $M$
and likewise smooth in the circle coordinate. If the circle is parametrized
to have unit speed by the coordinate $s\in \lbrack 0,S)$, then the $L^{2}$
inner products on basic functions and basic one-forms are defined by 
\begin{eqnarray*}
\left\langle f,g\right\rangle &=&\int_{0}^{S}f(s)g(s)v(s)\,ds, \\
\left\langle \alpha (s)\,ds,\beta (s)\,ds\right\rangle &=&\int_{0}^{S}\alpha
(s)\,\beta (s)\,v(s)\,ds.
\end{eqnarray*}
Note that $S$ is the transverse volume $V_{tr}$ of $(M,\mathcal{F})$,
defined as in Theorem~\ref{tracebasic}. We then compute that the basic
Laplacian on functions is given by 
\begin{equation*}
\Delta _{B}f=-\frac{\partial ^{2}}{\partial s\,^{2}}f-\frac{v^{\prime }}{v}%
\frac{\partial }{\partial s}f.
\end{equation*}
Since the foliation is transversally oriented, we may assume that we have
chosen a unit normal vector field $U=\frac{\partial }{\partial s}$. An
elementary calculation shows that the \textit{total mean curvature} $H(s)$
is given by 
\begin{equation}
H(s):=-\ell \int_{L_{s}}\left\langle \mathbf{H}(x),U(x)\right\rangle \,%
\mathrm{dvol}(x)=\frac{v^{\prime }(s)}{v(s)},  \label{totalmean}
\end{equation}
where $\mathbf{H}(x)$ is the mean curvature vector field of the leaf $L_{s}$
corresponding to the coordinate $s$ and $\ell $ is the dimension of each
leaf. Recall that $\mathbf{H}$ is defined as follows. Given $x\in L_{s}$,
let $\{E_{i}\}_{i=1}^{\ell }$ be an orthonormal basis of $T_{x}L$. We define 
\begin{equation*}
\mathbf{H}(x)=\frac{1}{\ell }\sum_{i=1}^{\ell }\left( \nabla
_{E_{i}}E_{i}\right) ^{\perp },
\end{equation*}
where $\perp $ denotes the projection onto the normal space. We denote the 
\textit{mean curvature} of $L_{s}$ by $h(x)=\Vert \mathbf{H}(x)\Vert $.

In summary, the basic Laplacian on functions $\psi :\left[ 0,V_{tr}\right]
\rightarrow \mathbb{R}$ is given by 
\begin{equation}
\Delta _{B}\psi (s)=-\frac{d^{2}}{ds^{2}}\psi (s)-H(s)\frac{d}{ds}\psi (s).
\label{basic1d}
\end{equation}
The trace of the basic heat kernel may now be computed in the standard way
from this operator on the circle. Let an arbitrary point on the circle be
denoted by the coordinate $0$, and let $x$ be any other point within $\frac{%
V_{tr}}{2}$ of $0$. Following the computation in \cite[pp.69--70]{Ro}, we
get the following asymptotic expansion of the basic heat kernel $%
K_{B}(t,x,0) $ as $t\rightarrow 0$: 
\begin{equation*}
K_{B}(t,x,0)\sim \frac{e^{-x^{2}/4t}}{\sqrt{4\pi t}}\left(
u_{0}(x)+u_{1}(x)t+u_{2}(x)t^{2}+\ldots \right) ,
\end{equation*}
where $u_{0}(x)=\frac{1}{R_{-1}(x)}$ and $u_{k+1}(x)$ for $k\geq 0$ is given
by 
\begin{equation}
R_{k}(x)u_{k+1}(x)=-\int_{0}^{x}\frac{R_{k}(y)}{y}\Delta _{B}u_{k}(y)\,dy,
\label{recurse1}
\end{equation}
where for every $j\geq -1$ 
\begin{eqnarray*}
R_{j}(x) &=&x^{j+1}\exp \left( \frac{1}{2}\int_{0}^{x}H(t)\,dt\right) \\
&=&x^{j+1}\sqrt{\frac{v(0)}{v(x)}}.
\end{eqnarray*}
Then (\ref{recurse1}) becomes 
\begin{equation}
x^{k+1}\sqrt{\frac{v(0)}{v(x)}}u_{k+1}(x)=-\int_{0}^{x}y^{k}\frac{v(0)}{v(x)}%
\Delta _{B}u_{k}(y)\,dy.  \label{recurse2}
\end{equation}
Therefore, 
\begin{equation*}
u_{0}(x)=\sqrt{\frac{v(x)}{v(0)}},
\end{equation*}
and 
\begin{eqnarray*}
u_{1}(x) &=&\frac{1}{x}\sqrt{\frac{v(x)}{v(0)}}\int_{0}^{x}\sqrt{\frac{v(0)}{%
v(y)}}\left( \left( \sqrt{\frac{v(y)}{v(0)}}\right) ^{\prime \prime
}+H(y)\left( \sqrt{\frac{v(y)}{v(0)}}\right) ^{\prime }\right) \,dy \\
&=&\frac{1}{4x}\sqrt{\frac{v(x)}{v(0)}}\int_{0}^{x}\left( \frac{v^{\prime }}{%
v}\right) ^{2}+2\frac{v^{\prime \prime }}{v}\,dy \\
&=&\frac{1}{4x}\sqrt{\frac{v(x)}{v(0)}}\int_{0}^{x}\left( 2H^{\prime
}(y)+3(H(y))^{2}\right) \,dy.
\end{eqnarray*}
Taking the limit as $x\rightarrow 0$, we obtain 
\begin{eqnarray*}
u_{0}(0) &=&1 \\
u_{1}(0) &=&\frac{1}{4}\left( 2H^{\prime }(0)+3(H(0))^{2}\right) .
\end{eqnarray*}
By realizing that the coordinate $0$ was labelled arbitrarily and by
integrating the above quantities over the circle, we obtain the following
theorem:

\begin{theorem}
\label{asympcodim1} Let $(M,\mathcal{F})$ be a transversally oriented
Riemannian foliation of codimension one without dense leaves. Then the trace 
$K_{B}(t)$ of the basic heat operator has the following asymptotic expansion
as $t\rightarrow 0$. For any nonnegative integer $J$, 
\begin{equation*}
K_{B}(t)=\frac{1}{\sqrt{4\pi t}}\left( A_{0}+A_{1}t+\ldots
+A_{J}t^{J}+O\left( t^{J+1}\right) \right) ,
\end{equation*}
where 
\begin{eqnarray*}
A_{0}=V_{tr}, \\
A_{1}=\frac{3}{4}\ell ^{2}\left( \Vert h\Vert _{2}\right) ^{2},
\end{eqnarray*}
and the other basic heat invariants may be computed using the recursion
formulas and integrations described above. Here, $V_{tr}$ is the transverse
volume of the foliation, and $\Vert h\Vert _{2}$ is the $L^{2}$ norm of the
mean curvature.
\end{theorem}

\begin{corollary}
Let $(M,\mathcal{F})$ be as in the theorem above. Then the spectrum of the
basic Laplacian on functions determines the $L^{2}$ norm of the mean
curvature. In particular, the foliation is minimal if and only if $A_{1}=0$
in Theorem~\ref{asympcodim1}.
\end{corollary}

\begin{remark}
The above theorem and corollary may be applied in cases of higher
codimension if all of the leaf closures are transversally oriented and have
codimension one.
\end{remark}

\subsection{Codimension One Foliations That Are Not Transversally Orientable}

\label{nonorientable} We now show how the results in the last section need
to be modified if $\left( M,\mathcal{F}\right) $ is not transversally
orientable. We will need the results of this section when we consider the
case of transversally orientable codimension two foliations whose leaf
closures are codimension one and are not necessarily transversally
orientable. If $\left( M,\mathcal{F}\right) $ is a codimension one
Riemannian foliation that is not transversally orientable, it has a double
cover $\left( \widetilde{M},\widetilde{\mathcal{F}}\right) $ that is
transversally orientable. Basic functions on $M$ correspond to basic
functions on $\widetilde{M}$ that are invariant under the orientation
reversing, isometric involution (the deck transformation). Thus the basic
Laplacian is a second order operator on a closed interval with Neumann
boundary conditions instead of a circle. Part of the analysis from the last
section is relevant, so that we obtain the following: 
\begin{eqnarray}
\Delta _{B}\psi (s) &=&-\frac{d^{2}}{ds^{2}}\psi (s)-H(s)\frac{d}{ds}\psi (s)
\notag \\
\psi ^{\prime }(0) &=&\psi ^{\prime }\left( V_{tr}\right) =0.  \label{BVP}
\end{eqnarray}
The asymptotics of the trace of the associated heat operator is a standard
problem. If $\widetilde{K}\left( t,\widetilde{s}_{1},\widetilde{s}
_{2}\right) $ is the lifted heat kernel to the circle, it corresponds to the
following differential operator on $\left[ -V_{tr},V_{tr}\right] $ $\ $with
periodic boundary conditions: 
\begin{equation*}
L\alpha \left( \widetilde{s}\right) =-\frac{d^{2}}{d\widetilde{s}^{2}}\alpha
(\widetilde{s})-\widetilde{H}(\widetilde{s})\frac{d}{d\widetilde{s}}\alpha (%
\widetilde{s}),  \label{liftedBVP}
\end{equation*}
where 
\begin{equation*}
\widetilde{H}(\widetilde{s})=\left\{ 
\begin{array}{ll}
H\left( \widetilde{s}\right) & \text{if \ }0<\widetilde{s}<V_{tr} \\ 
-H\left( -\widetilde{s}\right) & \text{if \ }-V_{tr}<\widetilde{s}<0%
\end{array}
\right. .
\end{equation*}
Note that $\widetilde{H}(\widetilde{s})$ is the logarithmic derivative of
the volume of the leaf on $\left( \widetilde{M},\widetilde{\mathcal{F}}%
\right) $ and thus extends to be a smooth function on the circle. In
particular, this implies that all even derivatives of $\ \widetilde{H}(%
\widetilde{s})$ and $H(\widetilde{s})$\ at $\widetilde{s}=0$ or $V_{tr}$ are
zero. A similar argument shows that the corresponding volume functions $%
v\left( \widetilde{s}\right) $and $\widetilde{v}\left( \widetilde{s}\right) $
have zero odd derivatives at $\widetilde{s}=0$ or $V_{tr}$. The heat kernel $%
K(t,s_{1},s_{2})$ for the original boundary value problem (\ref{BVP})
satisfies 
\begin{equation*}
K(t,s,s)=\widetilde{K}\left( t,s,s\right) +\widetilde{K}\left( t,s,-s\right)
,
\end{equation*}
We have that 
\begin{equation*}
\widetilde{K}\left( t,\widetilde{s}_{1},\widetilde{s}_{2}\right) \sim \frac{%
e^{-r^{2}/4t}}{\sqrt{4\pi t}}\left( u_{0}\left( \widetilde{s}_{1},\widetilde{%
s}_{2}\right) +u_{1}\left( \widetilde{s}_{1},\widetilde{s}_{2}\right)
t+\ldots \right) ,
\end{equation*}
where $r=$dist$\left( \widetilde{s}_{1},\widetilde{s}_{2}\right) $, and the
functions $u_{j}$ are explicitly computable from the differential equation (%
\ref{liftedBVP}). The trace is computed by the integral 
\begin{eqnarray*}
K_{B}(t) &=&\int_{0}^{V_{tr}}K(t,s,s)\,ds \\
&=&\int_{0}^{V_{tr}}\widetilde{K}\left( t,s,s\right) +\widetilde{K}\left(
t,s,-s\right) \,\,ds \\
&=&\frac{1}{2}\int_{-V_{tr}}^{V_{tr}}\widetilde{K}\left( t,s,s\right) +%
\widetilde{K}\left( t,s,-s\right) \,\,ds \\
&\sim &\frac{1}{\sqrt{4\pi t}}\left(
A_{0}+B_{0}t^{1/2}+A_{1}t+B_{1}t^{3/2}+\ldots \right) ,
\end{eqnarray*}
where $A_{j}$ is defined as in the oriented case, and $B_{j}$ depends on the
derivatives of $u_{j}$ evaluated at $(0,0)$ and $\left( V_{tr},V_{tr}\right) 
$. The first nontrivial coefficients in the formula are: 
\begin{eqnarray*}
A_{0} &=&\frac{1}{2}\int_{-V_{tr}}^{V_{tr}}u_{0}\left( \widetilde{s},%
\widetilde{s}\right) \,d\widetilde{s} \\
B_{0} &=&\frac{\sqrt{\pi }}{2}\left( u_{0}\left( 0,0\right) +u_{0}\left(
V_{tr},V_{tr}\right) \right) \\
A_{1} &=&\frac{1}{2}\int_{-V_{tr}}^{V_{tr}}u_{1}\left( \widetilde{s},%
\widetilde{s}\right) \,d\widetilde{s} \\
B_{1} &=&\frac{\sqrt{\pi }}{2}\left( u_{1}\left( 0,0\right) +u_{1}\left(
V_{tr},V_{tr}\right) \right) \\
&&+\frac{\sqrt{\pi }}{8}\left( \partial _{1}\partial _{1}-2\partial
_{1}\partial _{2}+\partial _{2}\partial _{2}\right) u_{0}\left( 0,0\right) \\
&&+\frac{\sqrt{\pi }}{8}\left( \partial _{1}\partial _{1}-2\partial
_{1}\partial _{2}+\partial _{2}\partial _{2}\right) u_{0}\left(
V_{tr},V_{tr}\right)
\end{eqnarray*}
From the calculations in the transversally oriented case, we have, 
\begin{eqnarray*}
u_{0}(s_{1},s_{2}) &=&\sqrt{\frac{v(s_{1})}{v(s_{2})}} \\
u_{1}(s_{1},s_{2}) &=&\frac{1}{4\left( s_{1}-s_{2}\right) }\sqrt{\frac{%
v(s_{1})}{v(s_{2})}}\int_{s_{2}}^{s_{1}}\left( 2\widetilde{H}^{\prime }(y)+3(%
\widetilde{H}(y))^{2}\right) \,dy,
\end{eqnarray*}
which after some calculation implies that 
\begin{eqnarray*}
u_{0}\left( \widetilde{s},\widetilde{s}\right) &=&1 \\
u_{1}\left( \widetilde{s},\widetilde{s}\right) &=&\frac{1}{4}\left( 2%
\widetilde{H}^{\prime }(\widetilde{s})+3(\widetilde{H}(\widetilde{s}%
))^{2}\right) \\
\left( \partial _{1}\partial _{1}-2\partial _{1}\partial _{2}+\partial
_{2}\partial _{2}\right) u_{0}\left( 0,0\right) &=&0 \\
\left( \partial _{1}\partial _{1}-2\partial _{1}\partial _{2}+\partial
_{2}\partial _{2}\right) u_{0}\left( V_{tr},V_{tr}\right) &=&0
\end{eqnarray*}
These equations imply that 
\begin{eqnarray*}
A_{0} &=&V_{tr} \\
B_{0} &=&\sqrt{\pi } \\
A_{1} &=&\frac{1}{8}\int_{-V_{tr}}^{V_{tr}}\left( 2\widetilde{H}^{\prime }(%
\widetilde{s})+3(\widetilde{H}(\widetilde{s}))^{2}\right) \,d\widetilde{s} \\
&=&\frac{3}{8}\int_{-V_{tr}}^{V_{tr}}(\widetilde{H}(\widetilde{s}))^{2}\,d%
\widetilde{s}=\frac{3}{4}\int_{0}^{V_{tr}}(H(s))^{2}\,ds \\
B_{1} &=&\frac{\sqrt{\pi }}{8}\left( 2\widetilde{H}^{\prime }(0)+3(%
\widetilde{H}(0))^{2}+2\widetilde{H}^{\prime }(V_{tr})+3(\widetilde{H}%
(V_{tr}))^{2}\right) \\
&=&\frac{\sqrt{\pi }}{4}\left( \widetilde{H}^{\prime }(0)+\widetilde{H}%
^{\prime }(V_{tr})\right) =\frac{\sqrt{\pi }}{4}\left( H^{\prime
}(0)+H^{\prime }(V_{tr})\right) .
\end{eqnarray*}
In summary, we have the following theorem.

\begin{theorem}
\label{codim1nonorientable} Let $(M,\mathcal{F})$ be a Riemannian foliation
of codimension one such that the leaves are not dense and the foliation is
not transversally orientable. Then the trace $K_{B}(t)$ of the basic heat
operator has the following asymptotic expansion as $t\rightarrow 0$. 
\begin{equation*}
K_{B}(t)=\frac{1}{\sqrt{4\pi t}}\left(
A_{0}+B_{0}t^{1/2}+A_{1}t+B_{1}t^{3/2}+...\right) ,
\end{equation*}%
where 
\begin{eqnarray*}
A_{0} &=&V_{tr}, \\
B_{0} &=&\sqrt{\pi } \\
A_{1} &=&\frac{3}{4}\ell ^{2}\left( \Vert h\Vert _{2}\right) ^{2}, \\
B_{1} &=&\frac{\sqrt{\pi }}{4}\left( F(0)+F(V_{tr})\right)
\end{eqnarray*}%
and the other basic heat invariants may be computed using the techniques
described above. Here, $V_{tr}$ is the transverse volume of the foliation, $%
\Vert h\Vert _{2}$ is the $L^{2}$ norm of the mean curvature, and $%
F(0)+F(V_{tr})$ is the sum of the second normal derivatives of the logarithm
of leaf volume, evaluated at the two leaves with $\mathbb{Z}_{2}$ holonomy
(this quantity is independent of the choice of normal at any point of these
leaves).
\end{theorem}

\begin{corollary}
Let $(M,\mathcal{F})$ a codimension one Riemannian foliation. Then the
spectrum of the basic Laplacian on functions determines whether or not the
leaves are dense. If the leaves are not dense, the spectrum also determines
whether or not the foliation is transversally orientable, the $L^{2}$ norm
of the mean curvature, and the average of the second normal derivatives of
the logarithm of leaf volume at the two leaves with $\mathbb{Z}_{2}$
holonomy in the nonorientable case. In particular, the foliation is minimal
if and only if the coefficient $A_{1}=0$.
\end{corollary}

\begin{remark}
The above theorem and corollary may be applied in cases of higher
codimension if all of the leaf closures have codimension one.
\end{remark}

\subsection{The General Case\label{general}}

We now prove some results that will be applied to the codimension two case
in Section~\ref{codim2}. These results are completely general and may be
used to compute the asymptotic expansion in special cases of arbitrary
codimension. We first review results that will be used in our computations.
Recall that $\widetilde{K}(t,w,v):=\widetilde{K_{W}}(t,w,v)$ is the heat
kernel corresponding to the operator $\widetilde{\Delta _{W}}$ defined in~(%
\ref{twistlap}), so that if $\varepsilon >0$ is sufficiently small, dist$%
(w,v)>\varepsilon $ implies that 
\begin{equation*}
\widetilde{K}(t,w,v)=O\left( e^{-c/t}\right)
\end{equation*}%
as $t\rightarrow 0$, for some constant $c$ (see, for example, \cite{Gr}). As
a consequence, the asymptotics of the integral in Proposition~\ref{trbasic1}
over $G\times W$ are the same as the asymptotics of the integral over $U$,
where $U$ is any arbitrarily small neighborhood of the compact subset $%
\{(g,w)\in G\times W\,|\,wg=w\}$, up to an error term of the form $O\left(
e^{-c/t}\right) $.

We will now decompose $W$ into pieces and use this decomposition to
partition a neighborhood of $\{(g,w)\in G\times W\,|\,wg=w\}$. Given an
orbit $X$ of $G$ and $w\in X$, $X$ is naturally diffeomorphic to $G/H_{w}$,
where $H_{w}=\{g\in G\,|wg=w\}$ is the (closed) isotropy subgroup. As we
mentioned before, $H_{w}$ is isomorphic to the structure group corresponding
to the principal bundle $\pi :\rho ^{-1}(w)\rightarrow \overline{L}$, where $%
\overline{L}$ is the leaf closure $\pi \left( \rho ^{-1}(w)\right) $ in $M$.
Given a subgroup $H$ of $G$, let $\left[ H\right] $ denote the conjugacy
class of $H$. The isotropy type of the orbit $X$ is defined to be the
conjugacy class $\left[ H_{w}\right] $, which is well--defined independent
of $w\in X$. There are a finite number of isotropy types of orbits in $W$
(see \cite[p.~173]{Bre}). We define the usual partial ordering (see \cite[%
p.~42]{Bre}) on the isotropy types by declaring that 
\begin{equation*}
\left[ H\right] \leq \left[ K\right] \iff H\text{ is conjugate to a subgroup
of }K.
\end{equation*}%
Let $\{\left[ H_{i}\right] :i=1,\dots ,r\}$ be the set of isotropy types
occurring in $W$, arranged so that 
\begin{equation*}
\left[ H_{i}\right] \leq \left[ H_{j}\right] \Longrightarrow i\geq j
\end{equation*}%
(see \cite[p. 51]{Kaw}). Let $W\left( \left[ H\right] \right) $ denote the
union of orbits of isotropy type $\left[ H\right] $ in $W$. The set $W\left( %
\left[ H_{i}\right] \right) $ is in general a $G$--invariant submanifold of $%
W$ (see \cite[p.~202]{Kaw}). Also, $W\left( \left[ H_{1}\right] \right) $ is
closed, and $W\left( \left[ H_{r}\right] \right) $ is open and dense in $W$ (%
\cite[p.~50,~216]{Kaw}). Thus, $W$ is the disjoint union of the submanifolds 
$W\left( \left[ H_{i}\right] \right) $ for $1\leq i\leq r$.

Now, given a proper, $G$-invariant submanifold $S$ of $W$ and $\varepsilon
>0 $, let $T_{\varepsilon }(S)$ denote the union of the images of the
exponential map at $s$ for $s\in S$ restricted to the ball of radius $%
\varepsilon $ in the normal bundle at $S$. It follows that $T_{\varepsilon
}(S)$ is also $G$-invariant. We now decompose $W$ as a disjoint union of
sets $W_{1},\dots ,W_{r}$. If there is only one isotropy type on $W$, then $%
r=1$, and we let $W_{1}=W$. Otherwise, let $W_{1}=T_{\varepsilon }\left(
W\left( \left[ H_{1}\right] \right) \right) $. For $1<j\leq r-1$, let 
\begin{equation*}
{\ W_{j}=T_{\varepsilon }\left( W\left( \left[ H_{j}\right] \right) \right)
\setminus \bigcup_{i=1}^{j-1}W_{i}}.
\end{equation*}%
Let 
\begin{equation*}
{\ W_{r}=W\setminus \bigcup_{i=1}^{r-1}W_{i}}.
\end{equation*}%
Clearly, $\varepsilon >0$ must be chosen sufficiently small in order for the
following lemma to be valid. We in addition insist that $\varepsilon $ be
chosen sufficiently small so that the asymptotic expansion for $\widetilde{%
K_{W}}(t,x,y)$ is valid if the distance from $x$ to $y$ is less than $%
\varepsilon $. The following facts about this decomposition are contained in 
\cite[pp.~203ff]{Kaw}:

\begin{lemma}
\label{decomposition}With $W$, $W_{i}$ defined as above, we have, for every $%
i\in \{1,\ldots ,r\}$:

\begin{enumerate}
\item $\displaystyle W=\bigcup_{i=1}^{r}W_{i}$.

\item $W_{i}$ is a union of $G$-orbits.

\item The closure of $W_{i}$ is a compact $G$-manifold with corners.

\item If $\left[ H_{j}\right] $ is the isotropy type of an orbit in $W_{i}$,
then $j\geq i$.

\item The distance between the submanifold $W\left( \left[ H_{j}\right]
\right) $ and $W_{i}$ for $j<i$ is at least $\varepsilon $.
\end{enumerate}
\end{lemma}

\begin{remark}
The lemma above remains true if at each stage $T_{\varepsilon }\left(
W\left( \left[ H_{j}\right] \right) \right) $ is replaced by any
sufficiently small open neighborhood of $W(H_{j})$ that contains \linebreak $%
T_{\varepsilon }\left( W\left( \left[ H_{j}\right] \right) \right) $, that
is a union of $G$-orbits, and whose closure is a manifold with corners.
\end{remark}

Therefore, by Proposition~\ref{trbasic1}, the trace of the basic heat kernel
is given by

\begin{equation}
K_{B}(t)=\sum_{i=1}^{r}\int_{W_{i}}\int_{G}\widetilde{K}(t,w,wg)\,\chi (g)\,%
\mathrm{dvol}_{W}(w).  \label{tracepiece}
\end{equation}

Let $H_{j}$ be the isotropy subgroup of $w\in W\left( \left[ H_{j}\right]
\right) $, and let $\gamma $ be a geodesic orthogonal to $W\left( \left[
H_{j}\right] \right) $ through $w$. This situation occurs exactly when this
geodesic is orthogonal both to the fixed point set $W^{H_{j}}$ of $H_{j}$
and to the orbit $wG$ of $G$ containing $w$. For any $h\in H_{j}$, right
multiplication by $h$ maps geodesics orthogonal to $W^{H_{j}}$ through $w$
to themselves and likewise maps geodesics orthogonal to $wG$ through $w$ to
themselves. Thus, the group $H_{j}$ acts orthogonally on the normal space to 
$w\in W\left( \left[ H_{j}\right] \right) $ by the differential of right
multiplication. Observe in addition that there are no fixed points for this
action; that is, there is no element of the normal space that is fixed by
every $h\in H_{j}$. If $G=SO\left( q\right) $ acts by orientation-preserving
isometries, then $H_{j}$ acts on the normal space in the same way. Since $%
H_{j}$ acts without fixed points, the codimension of $W\left( \left[ H_{j}%
\right] \right) $ is at least two in the orientation-preserving case.

On the other hand, if the transformation group $G$ in question is abelian
and acts by orientation-preserving isometries, then the representation
theory of abelian groups implies that the representation space must be
even-dimensional (see \cite[pp.~107--110]{BrotD}). In this case, we would
then conclude that the codimension of each $W\left( \left[ H_{j}\right]
\right) $ is even. We mention this for the following reason. Since $%
\widetilde{\Delta _{W}}$ commutes with the $SO(q)$ action on the basic
manifold $W$, the integrand $\widetilde{K}(t,w,wg)$ in Proposition~\ref%
{trbasic1} is a class function, so that Weyl's integration formula may be
used to rewrite the integral as an integral over a maximal torus $T$ of $%
SO(q)$ (see \cite[pp.~163]{BrotD}). \ Formula (\ref{tracepiece}) becomes

\begin{equation}
K_{B}(t)=\sum_{i=1}^{r}\int_{W_{i}}\int_{T}\widetilde{K}(t,w,wg)\,\eta (g)\,%
\mathrm{dvol}_{W}(w),  \label{toruspiece}
\end{equation}%
where $\eta (g)$ is the volume form on $T$ multiplied by a bounded function
of $g\in T$. In the above expression, we may take $W_{i}$ to be those
constructed using $G=T$. Therefore, each $H_{j}$ is a subgroup of the torus,
and the codimension of each $W\left( \left[ H_{i}\right] \right) $ is even
in the orientation-preserving case, as has been explained previously.

As $t\rightarrow 0$, we need only evaluate the asymptotics of the integrals
in (\ref{toruspiece}) on an arbitrarily small neighborhood of the set $%
\{(g,w)\in T\times W\,|\,wg=w\}$. By the construction of $W_{i}$, the
integral over $T$ may be replaced by an integral over a small neighborhood
of $H_{i}$ in $T$. This neighborhood may be described as $N_{\varepsilon
^{\prime }}\left( H_{i}\right) =\left\{ gh\,|\,h\in H_{i}\,,g\in
B_{\varepsilon ^{\prime }}\right\} $, where $B_{\varepsilon ^{\prime }}$ is
a ball of radius $\varepsilon ^{\prime }$ centered at the identity in $\exp
_{e}H_{i}^{\perp }$. Here, $H_{i}^{\perp }$ is the normal space to $%
H_{i}\subset T$ at the identity $e$, and $\exp _{e}$ is the exponential map $%
\exp _{e}:\mathfrak{t}\rightarrow T$. We have 
\begin{equation}
K_{B}(t)=\sum_{i=1}^{r}\int_{W_{i}}\int_{B_{\varepsilon ^{\prime }}\times
H_{i}}\widetilde{K}(t,w,wgh)\,\eta ^{\prime }(g,h)\,\mathrm{dvol}%
_{W}(w)+O\left( e^{-c/t}\right) .  \label{torus2}
\end{equation}%
We can now make this integral over the torus explicit. A maximal torus of $%
SO(q)$ has dimension $\left[ \frac{q}{2}\right] $, and we define $T$ in the
following way. Let $\Theta =\left( \theta _{1},\ldots ,\theta _{m}\right)
\in (-\pi ,\pi ]^{m}$. If $q=2m$, let 
\begin{equation*}
T=\left\{ \left. M(\Theta )\,\right\vert \,\theta _{j}\in (-\pi ,\pi ]\text{
for every }j\right\} ,
\end{equation*}%
where 
\begin{equation*}
M(\Theta )=\left( 
\begin{array}{ccccc}
\cos \theta _{1} & \sin \theta _{1} & \ldots  & 0 & 0 \\ 
-\sin \theta _{1} & \cos \theta _{1} & \ldots  & 0 & 0 \\ 
\vdots  & \vdots  & \ddots  & \vdots  & \vdots  \\ 
0 & 0 & \ldots  & \cos \theta _{m} & \sin \theta _{m} \\ 
0 & 0 & \ldots  & -\sin \theta _{m} & \cos \theta _{m}%
\end{array}%
\right) 
\end{equation*}%
If $q=2m+1$, $T$ is defined similarly with 
\begin{equation*}
M(\Theta )=\left( 
\begin{array}{cccccc}
\cos \theta _{1} & \sin \theta _{1} & \ldots  & 0 & 0 & 0 \\ 
-\sin \theta _{1} & \cos \theta _{1} & \ldots  & 0 & 0 & 0 \\ 
\vdots  & \vdots  & \ddots  & \vdots  & \vdots  & \vdots  \\ 
0 & 0 & \ldots  & \cos \theta _{m} & \sin \theta _{m} & 0 \\ 
0 & 0 & \ldots  & -\sin \theta _{m} & \cos \theta _{m} & 0 \\ 
0 & 0 & \ldots  & 0 & 0 & 1%
\end{array}%
\right) .
\end{equation*}%
Then we have the following formulas (\cite[pp. 171, 219--221]{BrotD}) for
the form $\eta $ in formula (\ref{toruspiece}): 
\begin{equation*}
\eta (\Theta )=\frac{1}{m!\,2^{m-1}}\prod \left( 1-e^{\pm i\theta _{k}\pm
i\theta _{l}}\right) \,\cdot \frac{d\Theta }{(2\pi )^{m}}
\end{equation*}%
if $q=2m$, and 
\begin{equation*}
\eta (\Theta )=\frac{1}{m!\,2^{m}}\prod \left( 1-e^{\pm i\theta _{k}\pm
i\theta _{l}}\right) \prod \left( 1-e^{\pm i\theta _{j}}\right) \,\cdot 
\frac{d\Theta }{(2\pi )^{m}}
\end{equation*}%
if $q=2m+1$. Simplifying, we get 
\begin{equation*}
\eta (\Theta )=\frac{2^{m^{2}-3m+1}}{m!\,\pi ^{m}}\left( \prod_{1\leq
k<l\leq m}\left( \cos \theta _{k}-\cos \theta _{l}\right) ^{2}\right)
\,d\Theta 
\end{equation*}%
if $q=2m$, and 
\begin{equation*}
\eta (\Theta )=\frac{2^{m^{2}-2m}}{m!\,\pi ^{m}}\left( \prod_{1\leq k<l\leq
m}\left( \cos \theta _{k}-\cos \theta _{l}\right) ^{2}\right) \,\left(
\prod_{1\leq j\leq m}\left( 1-\cos \theta _{j}\right) \right) \,d\Theta 
\end{equation*}%
if $q=2m+1$. Equation (\ref{toruspiece}) becomes 
\begin{equation}
K_{B}(t)=\sum_{i=1}^{r}\int_{W_{i}}\int_{\left( -\pi ,\pi \right) ^{m}}%
\widetilde{K}(t,w,w\,M\left( \Theta \right) )\,\,\eta (\Theta )\,\mathrm{dvol%
}_{W}(w),  \label{torustrace}
\end{equation}%
using the appropriate choice of $\eta (\Theta )$\ above. The explicit
description of $\,\eta (\Theta )$\ may be used to make equation (\ref{torus2}%
) more explicit as well, but additional calculations and a choice of
coordinates on $\exp _{e}H_{i}^{\perp }$ are necessary.

\subsection{Codimension Two Riemannian Foliations}

\label{codim2}

We now explicitly derive the coefficients in the asymptotic expansion of the
trace of the basic heat operator on transversally oriented Riemannian
foliations of codimension two. We have the following possibilities:

\begin{enumerate}
\item (The trivial case.) The closure of every leaf of $\left( M,\mathcal{F}%
\right) $ is the manifold $M$. In this case, the basic functions are
constants, and the basic heat operator is the identity. Therefore, the trace 
$K_{B}\left( t\right) $ of the basic heat operator satisfies $K_{B}\left(
t\right) =1$ for every $t$.

\item Every leaf of $\left( M,\mathcal{F}\right) $ is closed. Then the
results of Section \ref{finite} apply. The basic manifold $W$ is
three-dimensional. \ Given $w\in W$, the $SO\left( 2\right) $ orbit $X$ of $%
w $ is a circle.

\item Each leaf closure of \ $\left( M,\mathcal{F}\right) $ has codimension
one. If \ $\overline{\mathcal{F}}$ denotes the collection of leaf closures
of $\left( M,\mathcal{F}\right) $, $\left( M,\overline{\mathcal{F}}\right) $
is a Riemannian foliation of codimension one. Observe that the basic
functions, $L^{2}$ inner products, and basic Laplacians for $\left( M,%
\mathcal{F}\right) $ and $\left( M,\overline{\mathcal{F}}\right) $\ are the
same, so that we have reduced to the nontrivial codimension one cases. If
the leaf closure foliation is transversally orientable, see Section~\ref%
{codim1}. If the leaf closure foliation is not transversally orientable, see
Section~\ref{nonorientable}.

\item The leaf closures of \ $\left( M,\mathcal{F}\right) $ have minimum
codimension one, but some leaf closures have codimension two. This situation
is the most interesting case that arises. At least one of the orbits has
finite isotropy. Thus, the basic manifold $W$ is a closed two-manifold with
an \ $SO\left( 2\right) $ action whose orbits have two different dimensions.
The circular orbits correspond to the codimension one leaf closures, and the
(isolated) fixed points correspond to the codimension two leaf closures.
Because the group action yields a vector field on $W$ that has index $1$ at
each fixed point, the Euler characteristic of $W$ must be a positive
integer. Therefore, $W$ is a sphere or a projective plane; for simplicity we
consider only the case where $W=S^{2}$. The metric is a function of the
height (orbit) multiplied by the standard metric on the sphere.\ \ The space
of leaf closures of $\left( M,\mathcal{F}\right) $\ is a closed interval.
This case could be considered as a one-dimensional problem (as in Section~%
\ref{codim1}), but the analysis is quite difficult because the mean
curvature of the leaf closures goes to infinity at the leaves with infinite
holonomy. Instead, we use the approach of Section~\ref{general}. The
circular orbits have orbit type $\left( \left\{ e\right\} \right) $, and the
fixed points have orbit type $\left( SO\left( 2\right) \right) $. Since $%
\left( \left\{ e\right\} \right) \leq \left( SO\left( 2\right) \right) $, we
can decompose $W=W_{1}\cup W_{2}$ as in Lemma~\ref{decomposition}, where $%
W_{1}$ is the union of two metric $\varepsilon $-disks centered at the fixed
points and $W_{2}=W\setminus W_{1}$. The maximal torus of $SO\left( 2\right) 
$ is itself, and equation (\ref{torustrace}) shows that the trace of the
basic heat operator is 
\begin{eqnarray}
K_{B}(t) &=&\int_{W_{1}}\int_{-\pi }^{\pi }\widetilde{K}(t,w,w\,M\left(
\theta \right) )\,\frac{d\theta }{2\pi }\,\,\,\mathrm{dvol}_{W}(w)  \notag \\
&+&\int_{W_{2}}\int_{-\pi }^{\pi }\widetilde{K}(t,w,w\,M\left( \theta
\right) )\,\,\frac{d\theta }{2\pi }\,\,\mathrm{dvol}_{W}(w),
\label{sphereparts}
\end{eqnarray}
where $M\left( \theta \right) =\left( 
\begin{array}{cc}
\cos \theta & -\sin \theta \\ 
\sin \theta & \cos \theta%
\end{array}
\right) \allowbreak $.\newline
We now state a result from \cite{Ri2}:
\end{enumerate}

\begin{theorem}
(in \cite{Ri2}) \label{oldthm} Let $\psi :M\rightarrow \mathbb{R}$ be the
smooth basic function defined by setting $\psi (x)$ equal to the volume of
any leaf closure of $\widehat{M}$ which intersects the fiber $\pi ^{-1}(x)$.
Let $q_{x}$ denote the codimension of the leaf closure $\overline{L}_{x}$
containing $x$ in $M$. Then, as $t\rightarrow 0$, we have the following
asymptotic expansion for any positive integer $k$: 
\begin{equation*}
K_{B}(t,x,x)=\frac{1}{(4\pi t)^{{}^{q_{x}/2}}}\left(
a_{0}(x)+a_{1}(x)t+\ldots +a_{k}(x)t^{k}+O\left( t^{k+1}\right) \,\right) ,
\end{equation*}%
where 
\begin{equation*}
a_{i}(x)=\sum \frac{2^{l}K_{m_{1}\ldots m_{l}}^{jl}(x)\,I_{m_{1}\ldots
m_{l}}^{Q\,l}}{\mathrm{Vol}\left( \overline{L}_{x}\right) \pi ^{Q/2}}.
\end{equation*}%
The constants $I_{m_{1}\ldots m_{l}}^{Q\,l}$ are defined by
integrals, and the functions 
$K_{m_{1}\ldots m_{l}}^{jl}(x)$ are determined
by the methods in \cite{Ri2}. The integer $Q$ is codimension of the
intersection of the leaf closure containing $\hat{x}$ with $\pi ^{-1}(x)$ in 
$\pi ^{-1}(x)$. The functions $a_{i}(x)$ are determined by the local
geometry of the foliation at $x\in M$ and by $\psi (x)$. In particular, $%
a_{0}(x)=\frac{1}{\mathrm{Vol}\left( \overline{L}_{x}\right) }$, and 
\begin{equation*}
a_{1}(x)=\frac{1}{\mathrm{Vol}\left( \overline{L}_{x}\right) }\Biggm(\frac{%
S^{W}(\rho (\hat{x}))}{6}+\frac{\Delta _{B}\psi }{2\psi }\left( x\right) +%
\frac{3(d\psi ,d\psi )}{4\psi ^{2}}\left( x\right) 
\end{equation*}%
\begin{equation*}
+\frac{1}{6}C_{X}^{12}\mathrm{Ric}^{W}\left( \rho (\hat{x})\right) +\frac{1}{%
12}C^{34}C^{12}T\left( \rho (\hat{x})\right) +\frac{1}{6}C^{24}C^{13}T\left(
\rho (\hat{x})\right) -\frac{S^{X}(\rho (\hat{x}))}{3}\Biggm).
\end{equation*}%
The curvature terms are evaluated at any $\hat{x}\in \pi ^{-1}(x)$ and are
independent of that choice. In the above equation, $S^{X}$ and $S^{W}$
denote the scalar curvatures of $X=\rho \left( \pi ^{-1}(x)\right) $ and $W$%
, respectively, $\mathrm{Ric}^{W}$ is the Ricci curvature tensor on $W$, $%
C^{ab}$ denotes tensor contraction in the $a^{\text{th}}$ and $b^{\text{th}}$
slots, and the subscript $X$ means that the contraction is taken over the
tangent space to $X$ at $\rho (\hat{x})$. The $(0,4)$--tensor $T$ on $X$ is
defined by 
\begin{eqnarray*}
T(V,W,Y,Z) &=&\left\langle \nabla _{V}^{\perp }W,\nabla _{Y}^{\perp
}Z\right\rangle ^{\perp },\text{ and } \\
T_{kprt} &=&T\left( \partial _{k},\partial _{p},\partial _{r},\partial
_{t}\right) .
\end{eqnarray*}%
The symbol $\perp $ refers to the orthogonal complement of $T_{\rho (\hat{x}%
)}X$.
\end{theorem}

The asymptotics of the second term in (\ref{sphereparts}) can be found by
integrating the asymptotics in Theorem~\ref{oldthm}, because the
coefficients $a_{j}\left( x\right) $ and error estimates are bounded on the
set $\left\{ x\in M\,|\,\rho \left( \pi ^{-1}\left( x\right) \right) \subset
W_{2}\right\} $. \ We note that all of these coefficient functions arise
from such an integral on $W$. Explicitly, for $w\in W_{2}$, 
\begin{eqnarray*}
\int_{-\pi }^{\pi }\widetilde{K}(t,w,w\,M\left( \theta \right) )\,\,\frac{%
d\theta }{2\pi } &=&\frac{1}{\sqrt{4\pi t}}\bigm (a_{0}(w)+a_{1}(w)t \\
&+&\ldots +a_{k}(w)t^{k}+O\left( t^{k+1}\right) \,\bigm),
\end{eqnarray*}
where 
\begin{equation*}
a_{0}(w)=\frac{\phi \left( w\right) }{\mathrm{Vol}\left( \pi \left( \rho
^{-1}\left( w\right) \right) \right) }=\frac{1}{\mathrm{Vol}\left(
X_{w}\right) },
\end{equation*}
letting $\phi \left( w\right) $ be the volume of the leaf closure $\rho
^{-1}\left( w\right) $ in $\widehat{M}$ and letting $X=X_{w}$ be the orbit
of $w$ in $W$. Observe that we have also used (\ref{metricw}) to convert the
integral over $M$ to an integral over $W$. Similarly, $a_{1}(w)$ is
obtainable from the expression for $a_{1}$ in the theorem. Since the orbits
are one-dimensional and $W$ is two-dimensional, 
\begin{eqnarray*}
S^{X}\left( w\right) &=&0\text{ } \\
S^{W}\left( w\right) &=&2K\left( w\right) \text{ } \\
C_{X}^{12}\mathrm{Ric}^{W}(w) &=&\mathrm{Ric}^{W}\left( \sigma _{w},\sigma
_{w}\right) =K\left( w\right) \\
C^{24}C^{13}T(w) &=&C^{34}C^{12}T(w)=\left\Vert \mathbf{H}(w)\right\Vert
^{2}=\kappa \left( w\right) ^{2},
\end{eqnarray*}
where $K\left( w\right) $ is the Gauss curvature of the basic manifold at $w$
and $\kappa \left( w\right) $ is the geodesic curvature of the orbit at $w$.
\ Also, if $x\in \pi \left( \rho ^{-1}\left( w\right) \right) $ and $%
\widehat{x}\in \pi ^{-1}\left( x\right) $, 
\begin{eqnarray*}
\frac{\Delta _{B}\psi }{2\psi }\left( x\right) +\frac{3(d\psi ,d\psi )}{%
4\psi ^{2}}\left( x\right) &=&\frac{\pi ^{\ast }\Delta _{B}\psi }{2\pi
^{\ast }\psi }\left( \widehat{x}\right) +\frac{3(\pi ^{\ast }d\psi ,\pi
^{\ast }d\psi )}{4\pi ^{\ast }\psi ^{2}}\left( \widehat{x}\right) \\
&=&\frac{\widehat{\Delta _{B}}\pi ^{\ast }\psi }{2\pi ^{\ast }\psi }\left( 
\widehat{x}\right) +\frac{3(d\pi ^{\ast }\psi ,d\pi ^{\ast }\psi )}{4\pi
^{\ast }\psi ^{2}}\left( \widehat{x}\right) \\
&=&\frac{\widehat{\Delta _{B}}\rho ^{\ast }\phi }{2\rho ^{\ast }\phi }\left( 
\widehat{x}\right) +\frac{3(d\rho ^{\ast }\phi ,d\rho ^{\ast }\phi )}{4\rho
^{\ast }\phi ^{2}}\left( \widehat{x}\right) \\
&=&\frac{\rho ^{\ast }\widetilde{\Delta _{W}}\phi }{2\rho ^{\ast }\phi }%
\left( \widehat{x}\right) +\frac{3(d\rho ^{\ast }\phi ,d\rho ^{\ast }\phi )}{%
4\rho ^{\ast }\phi ^{2}}\left( \widehat{x}\right) \\
&=&\frac{\Delta _{W}\phi }{2\phi }\left( w\right) +\frac{(d\phi ,d\phi )}{%
4\phi ^{2}}\left( w\right) ,
\end{eqnarray*}
using the definition of $\widetilde{\Delta _{W}}$ (see Section~\ref{setup}).
\ Therefore, we now obtain 
\begin{equation*}
a_{1}(w)=\frac{1}{\mathrm{Vol}\left( X_{w}\right) }\left( \frac{1}{2}K\left(
w\right) +\frac{1}{4}\kappa \left( w\right) ^{2}+\frac{\Delta _{W}\phi }{%
2\phi }\left( w\right) +\frac{(d\phi ,d\phi )}{4\phi ^{2}}\left( w\right)
\right) ,
\end{equation*}
and 
\begin{eqnarray*}
\int_{W_{2}}\int_{-\pi }^{\pi }\widetilde{K}(t,w,w\,M\left( \theta \right)
)\,\,\frac{d\theta }{2\pi }\,\,\mathrm{dvol}_{W}(w) &=& \\
\frac{1}{\sqrt{4\pi t}}\int_{W_{2}}\bigm(a_{0}(w)+a_{1}(w)t &+&O\left(
t^{2}\right) \,\bigm)\mathrm{dvol}_{W}(w).
\end{eqnarray*}
In the expression above, note that we could take the limit as $\varepsilon
\rightarrow 0$ (that is, as $W_{2}\rightarrow W$) in every term except the
one involving the geodesic curvature $\kappa \left( w\right) $.

\begin{enumerate}
\item Next, we determine the asymptotics of \ the integral over $W_{1}$ in (%
\ref{sphereparts}). Recall that $W_{1}$ is the disjoint union of two metric $%
\varepsilon $-disks $D_{1}$ and $D_{2}$ surrounding the singular orbits. \
Choosing geodesic polar coordinates $\left( r,\gamma \right) $ around one of
the singular orbits, let $C\left( r\right) $ denote the length of the orbit
at radius $r$. Then $C\left( r\right) =2\pi r\left( 1+O\left( r^{2}\right)
\right) $. The metric on the basic manifold is 
\begin{equation*}
\left( g_{ij}\left( r,\gamma \right) \right) =\left( 
\begin{array}{cc}
1 & 0 \\ 
0 & \frac{C\left( r\right) ^{2}}{4\pi ^{2}}%
\end{array}
\right) ,
\end{equation*}
where $g_{11}\left( r,\gamma \right) =\left\langle \partial _{r},\partial
_{r}\right\rangle $ and so on. The integral over $D_{1}$ is 
\begin{eqnarray*}
\int_{^{D_{1}}}\int_{-\pi }^{\pi } &\ &\widetilde{K}(t,w,w\,M\left( \theta
\right) )\,\,\frac{d\theta }{2\pi }\,\,\mathrm{dvol}_{W}(w) \\
&=&\int_{r=0}^{\varepsilon }\int_{\gamma =-\pi }^{\pi }\int_{-\pi }^{\pi }%
\widetilde{K}(t,\left( r,\gamma \right) ,\left( r,\gamma \right) \,M\left(
\theta \right) )\,\,\frac{d\theta }{2\pi }\,\,d\gamma \frac{C\left( r\right) 
}{2\pi }\,dr \\
&=&\int_{r=0}^{\varepsilon }\int_{0}^{\pi }\widetilde{K}(t,\left( r,0\right)
,\left( r,0\right) \,M\left( \theta \right) )\,\,d\theta \,\frac{C\left(
r\right) }{\pi }\,dr,
\end{eqnarray*}
since the integrand is independent of $\gamma $ because of the isometric
action of $M\left( \theta \right) $. The Minakshisundaram-Pleijel expansion
of $\widetilde{K}(t,\left( r,0\right) ,\left( r,0\right) \,M\left( \theta
\right) )$ is of the form 
\begin{eqnarray*}
&\ &\widetilde{K}(t,\left( r,0\right) ,\left( r,0\right) \,M\left( \theta
\right) ) \\
&=&\frac{1}{4\pi t}e^{-D^{2}\left( r,\theta \right) /4t}\left( u_{0}\left(
r,\theta \right) +u_{1}\left( r,\theta \right) t+u_{2}\left( r,\theta
\right) t^{2}+O\left( t^{3}\right) \right) ,
\end{eqnarray*}
where $D\left( r,\theta \right) $ is geodesic distance between $\left(
r,0\right) $ and $\left( r,0\right) \,M\left( \theta \right) $. Each $%
u_{j}\left( r,\theta \right) $ is a smooth, bounded function on the $%
\varepsilon $-disk. Next, observe that $D\left( r,0\right) =0$, $D\left(
r,\pm \pi \right) =2r$, and $D^{2}\left( r,\theta \right) =r^{2}\left(
\left( 1-\cos \theta \right) ^{2}+\sin ^{2}\theta \right) +O\left(
r^{3}\right) =\allowbreak 2r^{2}\left( 1-\cos \theta \right) +O\left(
r^{3}\right) $ for small $r$. Using properties of the exponential map, one
could easily show that $D\left( r,\theta \right) $ increases in $\theta $ on 
$\left[ 0,\pi \right] $. Hence, after changing coordinates in $\theta $
(affecting the $u_{j}\left( r,\theta \right) $ by $O\left( r\right) $), we
have that 
\begin{eqnarray}
&\ &\int_{D_{1}}\int_{-\pi }^{\pi }\widetilde{K}(t,w,w\,M\left( \theta
\right) )\,\,\frac{d\theta }{2\pi }\,\,\mathrm{dvol}_{W}(w)  \notag \\
&=&\int_{r=0}^{\varepsilon }\int_{0}^{\pi }\widetilde{K}(t,\left( r,0\right)
,\left( r,0\right) \,M\left( \theta \right) )\,\,d\theta \,\frac{C\left(
r\right) }{\pi }\,dr  \notag \\
&=&\frac{1}{4\pi t}\int_{r=0}^{\varepsilon }\int_{0}^{\pi }e^{-\frac{r^{2}}{%
2t}\left( 1-\cos \theta \right) }\left( \overline{u}_{0}\left( r,\theta
\right) +\overline{u}_{1}\left( r,\theta \right) t+O\left( t^{2}\right)
\right) \,d\theta \,\frac{C\left( r\right) }{\pi }\,dr  \notag \\
&=&\frac{1}{4\pi t}\int_{r=0}^{\varepsilon }\int_{0}^{\pi }e^{-\frac{r^{2}}{%
2t}\left( 1-\cos \theta \right) }\left( b_{0}\left( r,\theta \right)
+b_{1}\left( r,\theta \right) t+O\left( t^{2}\right) \right) \,d\theta
\,2r\,dr,  \label{easympt}
\end{eqnarray}
where by symmetry we may assume that $\left. \frac{\partial ^{j}}{\partial
r^{j}}b_{m}\left( r,\theta \right) \right\vert _{r=0}=0$ for $j$ odd.

We now compute the asymptotics of some integrals that will be used our
calculation. Let 
\begin{equation*}
I_{1}\left( t,\varepsilon \right) =\int_{r=0}^{\varepsilon }\int_{0}^{\pi
}e^{-\frac{r^{2}}{2t}\left( 1-\cos \theta \right) }\,d\theta \,2r\,dr.
\end{equation*}
By Tonelli's theorem we change the order of integration, and we then
integrate to get 
\begin{equation*}
I_{1}\left( t,\varepsilon \right) =2\varepsilon ^{2}s\int_{0}^{\pi }\frac{%
\left( 1-\,e^{-\frac{1}{2s}\left( 1-\cos \theta \right) }\right) }{\left(
1-\cos \theta \right) }\,\,d\theta ,
\end{equation*}
letting $s=\frac{t}{\varepsilon ^{2}}$. By changing variables by $\cos
\theta =\frac{x-s}{x+s}=1-\frac{2s}{x+s}$, we obtain 
\begin{eqnarray*}
I_{1}\left( t,\varepsilon \right) &=&2\varepsilon ^{2}s\int_{0}^{\infty }%
\frac{\left( 1-\,e^{-\frac{1}{x+s}}\right) }{\frac{2s}{x+s}}\,\,\frac{\sqrt{s%
}\,dx}{\sqrt{x}\left( x+s\right) } \\
&=&\varepsilon ^{2}\sqrt{s}\int_{0}^{\infty }\frac{\left( 1-\,e^{-\frac{1}{%
x+s}}\right) }{\sqrt{x}}\,\,dx \\
&=&\varepsilon ^{2}\sqrt{s}\,F\left( s\right) .
\end{eqnarray*}
Using a Lebesgue convergence theorem argument, it can be shown that the
integral $F\left( s\right) $ is smooth in $s$ and can be differentiated
under the integral sign. In fact, it can be shown that 
\begin{equation*}
F\left( s\right) =\int_{0}^{\infty }\frac{\left( 1-\,e^{-\frac{1}{x+s}%
}\right) }{\sqrt{x}}\,\,dx=\allowbreak \left( 2\sqrt{\pi }\right) +\left( -%
\frac{1}{2}\sqrt{\pi }\right) s+O\left( s^{2}\right) \allowbreak ,
\end{equation*}
so that 
\begin{eqnarray}
I_{1}\left( t,\varepsilon \right) &=&\varepsilon ^{2}\sqrt{\frac{t}{%
\varepsilon ^{2}}}\,F\left( \frac{t}{\varepsilon ^{2}}\right) =\varepsilon 
\sqrt{t}\,F\left( \frac{t}{\varepsilon ^{2}}\right)  \notag \\
&=&\varepsilon \sqrt{4\pi t}\left( 1-\frac{1}{4}\left( \frac{t}{\varepsilon
^{2}}\right) +O\left( \left( \frac{t}{\varepsilon ^{2}}\right) ^{2}\right)
\right) .  \label{I1}
\end{eqnarray}

Next, for any smooth, bounded function $h$ on the $\varepsilon $-disk,
consider the integral 
\begin{equation*}
I_{2}\left( t,\varepsilon ,h\right) =\int_{r=0}^{\varepsilon }\int_{0}^{\pi
}e^{-\frac{r^{2}}{2t}\left( 1-\cos \theta \right) }r^{2}h\left( r,\theta
\right) \,d\theta \,2r\,dr.
\end{equation*}
Replacing $r$ with $\varepsilon r$, we obtain 
\begin{equation*}
I_{2}\left( t,\varepsilon ,h\right) =\varepsilon
^{4}\int_{r=0}^{1}\int_{0}^{\pi }e^{-\frac{\varepsilon ^{2}r^{2}}{2t}\left(
1-\cos \theta \right) }r^{2}h\left( \varepsilon r,\theta \right) \,d\theta
\,2r\,dr.
\end{equation*}
As before, we let $s=\frac{t}{\varepsilon ^{2}}$, and we substitute $\cos
\theta =\frac{x-s}{x+s}=1-\frac{2s}{x+s}$. 
\begin{eqnarray*}
I_{2}\left( t,\varepsilon ,h\right) &=&\varepsilon
^{4}\int_{r=0}^{1}\int_{0}^{\infty }e^{-\frac{r^{2}}{x+s}}r^{2}h\left(
\varepsilon r,\theta \right) \,\frac{\sqrt{s}\,dx}{\sqrt{x}\left( x+s\right) 
}\,2r\,dr \\
&=&\varepsilon ^{4}\sqrt{s}\int_{0}^{\infty }\int_{r=0}^{1}e^{-\frac{r^{2}}{%
x+s}}h\left( \varepsilon r,\theta \right) \,\,2r^{3}\,dr\,\frac{\,dx}{\sqrt{x%
}\left( x+s\right) } \\
&=&\varepsilon ^{4}\sqrt{s}\,G\left( s,\varepsilon \right) ,
\end{eqnarray*}
where we have used the fact that the integral converges absolutely and
Fubini's theorem. Again, using a Lebesgue convergence theorem argument, it
can be shown that the integral $G\left( s,\varepsilon \right) $ is smooth in 
$s$ and can be differentiated under the integral sign. Moreover, for any
positive integer $m$, 
\begin{equation*}
G\left( s,\varepsilon \right) =c_{0}\left( \varepsilon \right) +c_{1}\left(
\varepsilon \right) s+c_{2}\left( \varepsilon \right) s^{2}+...+c_{m}\left(
\varepsilon \right) s^{m}+O\left( s^{m+1}\right) ,
\end{equation*}
where each $c_{j}\left( \varepsilon \right) $ remains bounded as $%
\varepsilon \rightarrow 0$. Therefore, 
\begin{eqnarray}
I_{2}\left( t,\varepsilon ,h\right) &=&\varepsilon ^{4}\sqrt{\frac{t}{%
\varepsilon ^{2}}}\,G\left( \frac{t}{\varepsilon ^{2}},\varepsilon \right)
=\varepsilon ^{3}\sqrt{t}\,G\left( \frac{t}{\varepsilon ^{2}},\varepsilon
\right)  \notag \\
&=&\varepsilon ^{3}\sqrt{4\pi t}\left( \overline{c}_{0}\left( \varepsilon
\right) +\overline{c}_{1}\left( \varepsilon \right) \frac{t}{\varepsilon ^{2}%
}+O\left( \left( \frac{t}{\varepsilon ^{2}}\right) ^{2}\right) \right) ,
\label{I2}
\end{eqnarray}
where $\overline{c}_{0}\left( \varepsilon \right) $ and $\overline{c}%
_{1}\left( \varepsilon \right) $ are bounded as $\varepsilon \rightarrow 0$.

We will now compute the asymptotics of the integral over $D_{1}$ in equation
(\ref{easympt}). We have that 
\begin{eqnarray*}
&\ &\int_{D_{1}}\int_{-\pi }^{\pi }\widetilde{K}(t,w,w\,M\left( \theta
\right) )\,\,\frac{d\theta }{2\pi }\,\,\mathrm{dvol}_{W}(w) \\
&=&\frac{1}{4\pi t}\int_{r=0}^{\varepsilon }\int_{0}^{\pi }e^{-\frac{r^{2}}{%
2t}\left( 1-\cos \theta \right) }\left( b_{0}\left( r,\theta \right)
+b_{1}\left( r,\theta \right) t+O\left( t^{2}\right) \right) \,d\theta
\,2r\,dr \\
&=&\frac{1}{4\pi t}\int_{r=0}^{\varepsilon }\int_{0}^{\pi }e^{-\frac{r^{2}}{%
2t}\left( 1-\cos \theta \right) }\left( b_{0}+O\left( r^{2}\right)
+b_{1}t+O\left( r^{2}t\right) +O\left( t^{2}\right) \right) \,d\theta
\,2r\,dr \\
&=&\frac{1}{4\pi t}\left( I_{1}\left( t,\varepsilon \right) \left(
b_{0}+b_{1}t\right) +I_{2}\left( t,\varepsilon ,h_{0}\right) +I_{2}\left(
t,\varepsilon ,h_{1}\right) t+O\left( I_{2}\left( t,\varepsilon
,h_{2}\right) t^{2}\right) \right)
\end{eqnarray*}
for some appropriate choices of smooth, bounded functions $h_{0}$, $h_{1}$,
and $h_{2}$. Note that we have denoted $b_{j}=\left. b_{j}\left( r,\theta
\right) \right\vert _{r=0}$. Substituting the expressions (\ref{I1}) and (%
\ref{I2}), \ we get 
\begin{eqnarray*}
&\ &\int_{D_{1}}\int_{-\pi }^{\pi }\widetilde{K}(t,w,w\,M\left( \theta
\right) )\,\,\frac{d\theta }{2\pi }\,\,\mathrm{dvol}_{W}(w) \\
&=&\frac{1}{\sqrt{4\pi t}}\biggm(\left( \varepsilon b_{0}+O\left(
\varepsilon ^{3}\right) \right) +\bigm(-\frac{b_{0}}{4\varepsilon }%
+\varepsilon b_{1} \\
&\ &+O\left( \varepsilon \right) \bigm)t+O\left( t^{2}\right) \biggm).
\end{eqnarray*}
We observe that by the above construction starting with the asymptotic
expansion of $\widetilde{K}(t,\left( r,0\right) ,\left( r,0\right) \,M\left(
\theta \right) )$, we have $b_{0}=1$. This implies that 
\begin{eqnarray*}
&\ &\int_{D_{1}}\int_{-\pi }^{\pi }\widetilde{K}(t,w,w\,M\left( \theta
\right) )\,\,\frac{d\theta }{2\pi }\,\,\mathrm{dvol}_{W}(w) \\
&=&\frac{1}{\sqrt{4\pi t}}\left( O\left( \varepsilon \right) +\left( -\frac{1%
}{4\varepsilon }+O\left( \varepsilon \right) \right) t+O\left( t^{2}\right)
\right) .
\end{eqnarray*}
We have a similar formula for the asymptotics of the integral over the other 
$\varepsilon $-disk $D_{2}$. Thus, the integral over $W_{1}$ satisfies 
\begin{eqnarray*}
&\ &\int_{W_{1}}\int_{-\pi }^{\pi }\widetilde{K}(t,w,w\,M\left( \theta
\right) )\,\,\frac{d\theta }{2\pi }\,\,\mathrm{dvol}_{W}(w) \\
&=&\frac{1}{\sqrt{4\pi t}}\left( O\left( \varepsilon \right) +\left( -\frac{1%
}{2\varepsilon }+O\left( \varepsilon \right) \right) t+O\left( t^{2}\right)
\right)
\end{eqnarray*}
A simple calculation shows that the $-\frac{1}{2\varepsilon }t$ term exactly
counteracts the blowing up of the term $a_{1}\left( w\right) t$ in (\ref%
{easympt}) as $\varepsilon \rightarrow 0$, as expected. We remark that the
analysis in this section could be extended to find the coefficients of
larger powers of $t$ in the obvious way. Putting these results together, we
have the following theorem.

\begin{theorem}
\label{2d1d}Suppose that $(M,\mathcal{F})$ is a Riemannian foliation of
codimension two, and suppose that the leaf closures of \ $\left( M,\mathcal{F%
}\right) $ have minimum codimension one, but some leaf closures have
codimension two. Then the basic manifold $W$ is a sphere, and the $SO\left(
2\right) $-action has exactly two fixed points $w_{1}$ and $w_{2}$. \ Let $%
W_{\varepsilon }=\left\{ w\in W\,|\,\text{dist}\left( w,w_{j}\right)
>\varepsilon \text{ for }j=1,2\right\} $. As $t\rightarrow 0$, the trace $%
K_{B}(t)$ of the basic heat kernel on functions satisfies the following
asymptotic expansion for any positive integer $J$: 
\begin{equation*}
K_{B}(t)=\frac{1}{\sqrt{4\pi t}}\left( A_{0}+A_{1}t+A_{2}t^{2}+\ldots
+A_{J}t^{J}+O\left( t^{{J+1}}\right) \right) ,
\end{equation*}
where 
\begin{equation*}
A_{0}=\int_{M}\frac{1}{\mathrm{Vol}\left( \overline{L_{x}}\right) }\,{%
\mathrm{dvol}}(x)
\end{equation*}
and in general $A_{j}=\lim_{\varepsilon \rightarrow 0}\left(
\int_{W_{\varepsilon }}\overline{a}_{j}(w)\,\mathrm{dvol}(w)-f_{j}\left( 
\frac{1}{\varepsilon }\right) \right) $ for a specific polynomial $f_{j}$
and where $\overline{a}_{j}(w)=\phi \left( w\right) \,a_{j}\left( \pi \left(
\rho ^{-1}\left( w\right) \right) \right) $ with $a_{j}$ as in Theorem~\ref%
{oldthm}. As before, $\overline{L_{x}}$ denotes the leaf closure containing $%
x\in M$. Specifically, 
\begin{equation*}
A_{1}=\lim_{\varepsilon \rightarrow 0}\left( \int_{W_{\varepsilon }}\frac{1}{%
\mathrm{Vol}\left( X_{w}\right) }S(w)\,\mathrm{dvol}(w)-\frac{1}{%
2\varepsilon }\right)
\end{equation*}
for 
\begin{equation*}
S(w)=\frac{1}{2}K\left( w\right) +\frac{1}{4}\kappa \left( w\right) ^{2}+%
\frac{\Delta _{W}\phi }{2\phi }\left( w\right) +\frac{(d\phi ,d\phi )}{4\phi
^{2}}\left( w\right) ,
\end{equation*}
where $\phi \left( w\right) $ is the volume of the leaf closure $\rho
^{-1}\left( w\right) $ in $\widehat{M}$ , $X_{w}$ is the orbit of $w$ in $W$%
, $K\left( w\right) $ is the Gauss curvature of the basic manifold at $w$,
and $\kappa \left( w\right) $ is the geodesic curvature of the orbit at $w$.
\end{theorem}
\end{enumerate}

\section{Examples}

\label{example}

In this section we compute two specific examples that demonstrate the
behavior described in the last section. The first example is a transversally
oriented, codimension two Riemannian foliation in which some of the leaf
closures have codimension one and others have codimension two.

\begin{example}
\label{sphereexample} Consider the one-dimensional foliation obtained by
suspending an irrational rotation on the standard unit sphere $S^{2}$. On $%
S^{2}$ we use the cylindrical coordinates $\left( z,\theta \right) $,
related to the standard rectangular coordinates by $x^{\prime }=\sqrt{\left(
1-z^{2}\right) }\cos \theta $, $y^{\prime }=\sqrt{\left( 1-z^{2}\right) }
\sin \theta $, $z^{\prime }=z$. Let $\alpha $ be an irrational multiple of $%
2\pi $, and let the three-manifold $M=S^{2}\times \left[ 0,1\right] /\sim $,
where $\left( z,\theta ,0\right) \sim \left( z,\theta +\alpha ,1\right) $. \
Endow $M$ with the product metric on $T_{z,\theta ,t}M\cong T_{z,\theta
}S^{2}\times T_{t}\mathbb{R}$. Let the foliation $\mathcal{F}$ be defined by
the immersed submanifolds $L_{z,\theta }=\cup _{n\in \mathbb{Z}}\left\{
z\right\} \times \left\{ \theta +\alpha \right\} \times \left[ 0,1\right] $
(not unique in $\theta $). The leaf closures $\overline{L}_{z}$ for $|z|<1$
are two-dimensional, and the closures corresponding to the poles ($z=\pm 1$)
are one-dimensional. Therefore, this foliation satisfies the hypothesis of
Theorem~\ref{2d1d}.

In \cite{Ri2}, we used this example to demonstrate the asymptotic behavior
of \linebreak
$K_{B}\left( t,z,z\right) $ for different values of $z$. We state the
results of some of the computations in \cite{Ri2}. The basic functions are
functions of $\ z$ alone, and the basic Laplacian on functions is $\Delta
_{B}=-\left( 1-z^{2}\right) \,\partial _{z}^{2}+2z\,\partial _{z}$. \ The
volume form on $M$ is $dz\,d\theta \,dt$, and the volume of the leaf closure
at $\ z$ is $\frac{1}{2\pi \sqrt{1-z^{2}}}$ for $|z|<1$. The eigenfunctions
are the Legendre polynomials $P_{n}\left( z\right) $ corresponding to
eigenvalues $m\left( m+1\right) $ for $m\geq 0$. From this information
alone, one may calculate that the trace $K_{B}\left( t\right) $ of the basic
heat operator is 
\begin{equation}
K_{B}\left( t\right) =\sum_{m\geq 0}e^{-m\left( m+1\right) t}=\frac{1}{\sqrt{%
4\pi t}}\left( \pi +\frac{\pi }{4}t+O\left( t^{2}\right) \right) .
\label{tracesuspense}
\end{equation}

The basic manifold $W$ corresponding to this foliation is a sphere with
points described by orthogonal coordinates $\left( z,\varphi \right) \in
\lbrack -1,1]\times \left( -\pi ,\pi \right] $. As shown in \cite{Ri2}, the
metric on $W$ is given by $\left\langle \partial _{z},\partial
_{z}\right\rangle =\frac{1}{1-z^{2}},\,\,\left\langle \partial _{\varphi
},\partial _{\varphi }\right\rangle =\frac{4\pi ^{2}\left( 1-z^{2}\right) }{%
4\pi ^{2}\left( 1-z^{2}\right) +z^{2}}$. The following geometric quantities
can be calculated from this data (see \cite{Ri2}): 
\begin{eqnarray*}
\mathrm{dvol}_{z,\varphi } &=&\frac{2\pi }{\sqrt{4\pi ^{2}\left(
1-z^{2}\right) +z^{2}}}\,dz\,d\varphi \\
\mathrm{Vol}\left( X_{z}\right) &=&\frac{4\pi ^{2}\sqrt{\left(
1-z^{2}\right) }}{\sqrt{4\pi ^{2}\left( 1-z^{2}\right) +z^{2}}} \\
K\left( z,\varphi \right) &=&2\frac{\left( 4\pi ^{2}-1\right) z^{2}+2\pi ^{2}%
}{\left( 4\pi ^{2}\left( 1-z^{2}\right) +z^{2}\right) ^{2}} \\
\frac{\Delta _{W}\phi }{2\phi }\left( z,\varphi \right) +\frac{(d\phi ,d\phi
)}{4\phi ^{2}}\left( z,\varphi \right) &=&\frac{\left( 4\pi ^{2}-1\right)
\left( \left( -3\pi ^{2}-\frac{3}{4}\right) z^{2}+\left( \pi ^{2}-\frac{1}{4}%
\right) z^{4}+2\pi ^{2}\right) }{\left( 4\pi ^{2}\left( 1-z^{2}\right)
+z^{2}\right) ^{2}} \\
\kappa \left( z,\varphi \right) ^{2} &=&\frac{z^{2}}{\left( 1-z^{2}\right)
\left( 4\pi ^{2}\left( 1-z^{2}\right) +z^{2}\right) ^{2}} \\
\varepsilon &=&\text{dist}\left( z=1,z=1-\varepsilon ^{\prime }\right) =%
\frac{\pi }{2}-\arcsin \left( 1-\varepsilon ^{\prime }\right)
\end{eqnarray*}
Using these computations, we may now compute $A_{0}$ and $A_{1}$ in Theorem~%
\ref{2d1d}. 
\begin{equation*}
A_{0}=\int_{M}\frac{1}{\mathrm{Vol}\left( \overline{L_{z}}\right) }\,{%
\mathrm{dvol}}(z,\theta ,t)=\int_{t=0}^{1}\int_{\theta =-\pi }^{\pi
}\int_{z=-1}^{1}\frac{1}{2\pi \sqrt{1-z^{2}}}\,dz\,d\theta \,dt=\pi .
\end{equation*}
Next, we substitute the functions into the formula for $A_{1}$ and simplify: 
\begin{eqnarray*}
A_{1} &=&\lim_{\varepsilon \rightarrow 0}\left( \int_{W_{\varepsilon }}\frac{%
1}{\mathrm{Vol}\left( X_{w}\right) }S(w)\,\mathrm{dvol}(w)-\frac{1}{%
2\varepsilon }\right) \\
&=&\lim_{\varepsilon ^{\prime }\rightarrow 0}\left( \int_{\varphi =-\pi
}^{\pi }\int_{z=-1+\varepsilon ^{\prime }}^{1-\varepsilon ^{\prime }}\frac{1%
}{8\pi }\frac{2-z^{2}}{\left( 1-z^{2}\right) ^{3/2}}\,\,dz\,d\varphi -\frac{1%
}{2\left( \frac{\pi }{2}-\arcsin \left( 1-\varepsilon ^{\prime }\right)
\right) }\right) \\
&=&\lim_{\varepsilon ^{\prime }\rightarrow 0}\left( \int_{z=-1+\varepsilon
^{\prime }}^{1-\varepsilon ^{\prime }}\,\,\frac{1}{4}\frac{2-z^{2}}{\left(
1-z^{2}\right) ^{3/2}}\,dz\,-\frac{1}{2\left( \frac{\pi }{2}-\arcsin \left(
1-\varepsilon ^{\prime }\right) \right) }\right) \\
&=&\lim_{\varepsilon ^{\prime }\rightarrow 0}\left( \,\,\allowbreak \left. 
\frac{1}{4}\left( \frac{z}{\sqrt{\left( 1-z^{2}\right) }}+\arcsin z\right)
\right\vert _{-1+\varepsilon ^{\prime }}^{1-\varepsilon ^{\prime }}\,-\frac{1%
}{2\left( \frac{\pi }{2}-\arcsin \left( 1-\varepsilon ^{\prime }\right)
\right) }\right) \\
&=&\lim_{\varepsilon ^{\prime }\rightarrow 0}\Bigm(\,\,\allowbreak \frac{1}{2%
}\left( \frac{1-\varepsilon ^{\prime }}{\sqrt{\left( 1-\left( 1-\varepsilon
^{\prime }\right) ^{2}\right) }}+\arcsin \left( 1-\varepsilon ^{\prime
}\right) \right) \\
&\ &-\frac{1}{2\left( \frac{\pi }{2}-\arcsin \left( 1-\varepsilon ^{\prime
}\right) \right) }\Bigm) \\
&=&\lim_{\varepsilon ^{\prime }\rightarrow 0}\left( \frac{\pi }{4}+O\left( 
\sqrt{\varepsilon ^{\prime }}\right) \right) =\frac{\pi }{4}.
\end{eqnarray*}
Therefore, Theorem~\ref{2d1d} implies that 
\begin{equation*}
K_{B}\left( t\right) =\frac{1}{\sqrt{4\pi t}}\left( \pi +\frac{\pi }{4}%
t+O\left( t^{2}\right) \right) ,
\end{equation*}
which agrees with the direct calculation (\ref{tracesuspense}).
\end{example}

The next example is an example of a codimension two, transversally oriented
Riemannian foliation such that not all of the leaf closures are
transversally orientable.

\begin{example}
\label{exnonorientable} This foliation is the suspension of an irrational
rotation of the flat torus and a $\mathbb{Z}_{2}$-action. Let $X$ be any
closed Riemannian manifold such that $\pi _{1}(X)=\mathbb{Z}\ast \mathbb{Z}$
--- the free group on two generators $\{\alpha ,\beta \}$. We normalize the
volume of $X$ to be 1. Let $\widetilde{X}$ be the universal cover. We define 
$M=\widetilde{X}\times S^{1}\times S^{1}\slash \pi _{1}(X)$, where $\pi
_{1}(X)$ acts by deck transformations on $\widetilde{X}$ and by $\alpha
\left( \theta ,\phi \right) =\left( 2\pi -\theta ,2\pi -\phi \right) $ and $%
\beta \left( \theta ,\phi \right) =\left( \theta ,\phi +\sqrt{2}\pi \right) $
on $S^{1}\times S^{1}$. We use the standard product-type metric. The leaves
of $\mathcal{F}$ are defined to be sets of the form $\left\{ (x,\theta ,\phi
)_{\sim }\,|\,x\in \widetilde{X}\right\} $. Note that the foliation is
transversally oriented, but the codimension one leaf closure foliation is
not transversally orientable. The leaf closures are sets of the form 
\begin{equation*}
\overline{L}_{\theta }=\left\{ (x,\theta ,\phi )_{\sim }\,|\,x\in \widetilde{%
X},\phi \in \lbrack 0,2\pi ]\right\} \bigcup \left\{ (x,2\pi -\theta ,\phi
)_{\sim }\,|\,x\in \widetilde{X},\phi \in \lbrack 0,2\pi ]\right\}
\end{equation*}
The basic functions and one-forms are: 
\begin{eqnarray*}
\Omega _{B}^{0} &=&\left\{ f\left( \theta \right) \right\} \\
\Omega _{B}^{1} &=&\left\{ g_{1}\left( \theta \right) d\theta +g_{2}(\theta
)d\phi \right\} ,
\end{eqnarray*}
where the functions are smooth and satisfy 
\begin{eqnarray*}
f\left( 2\pi -\theta \right) &=&f\left( \theta \right) \\
g_{i}\left( 2\pi -\theta \right) &=&-g_{i}\left( \theta \right)
\end{eqnarray*}
From this information, we calculate that $\Delta _{B}f\left( \theta \right)
=-f^{\prime \prime }\left( \theta \right) $. The eigenvalues $\left\{
n\,|\,n\geq 0\right\} $ correspond to the eigenfunctions $\left\{ \cos
n\theta \,|\,n\geq 0\right\} $. Then 
\begin{eqnarray*}
K_{B}\left( t,\theta ,\theta \right) &=&\frac{1}{2\pi }+\frac{1}{\pi }%
\sum_{n\geq 1}e^{-n^{2}t}\cos ^{2}\left( n\theta \right) \\
&=&\frac{1}{2\pi }\sum_{n\in \mathbb{Z}}e^{-n^{2}t}\cos ^{2}\left( n\theta
\right) \\
&=&\frac{1}{4\pi }\sum_{n\in \mathbb{Z}}e^{-n^{2}t}+\frac{1}{4\pi }%
\sum_{n\in \mathbb{Z}}e^{-n^{2}t}\cos \left( 2n\theta \right) \\
&=&\frac{1}{4\pi }\mathrm{tr}\left( e^{-t\Delta }\text{ on }L^{2}\left(
S^{1}\right) \right) +\frac{1}{2}K\left( t,0,2\theta \right) ,
\end{eqnarray*}
where $K\left( t,\theta _{1},\theta _{2}\right) $ is the heat kernel on
functions on $S^{1}$. Substituting the expressions for this known kernel and
its trace, we obtain 
\begin{multline*}
K_{B}\left( t,\theta ,\theta \right) =\frac{1}{4\pi }\sqrt{\frac{\pi }{t}}%
\sum_{k\in \mathbb{Z}}e^{-k^{2}\pi ^{2}\slash t}+\frac{1}{2}\frac{1}{\sqrt{%
4\pi t}}\sum_{k\in \mathbb{Z}}e^{-\left( 2\theta +2k\pi \right) ^{2}\slash
4t} \\
=\left\{ 
\begin{array}{ll}
\frac{1}{2\sqrt{\pi t}}+\mathcal{O}\left( e^{-\pi ^{2}/t}\right) & \text{if }%
\theta =k\pi \text{ for some }k\in \mathbb{Z} \\ 
\frac{1}{4\sqrt{\pi t}}+\mathcal{O}\left( e^{-c\left( \theta \right)
/t}\right)  & \text{otherwise, for }c\left( \theta \right) =\min \left\{
\left. \left( \theta +k\pi \right) ^{2}\right\vert ~k\in \mathbb{Z}\right\}%
\end{array}
\right.
\end{multline*}
\newline
The basic manifold $\widehat{W}$ is an $SO(2)$-manifold, defined by $%
\widehat{W}=[0,\pi ]\times S^{1}\slash \sim $, where the circle has length $%
1$ and $\left( \theta =0\text{ or }\pi ,\gamma \right) \sim \left( \theta =0%
\text{ or }\pi ,-\gamma \right) $. This is a Klein bottle, since it is the
connected sum of two projective planes. The group $SO(2)$ acts on $\widehat{W%
}$ via the usual action on $S^{1}$. It is a simple exercise to calculate the
trace of the basic heat operator from the eigenvalues alone: 
\begin{equation*}
K_{B}\left( t\right) =\mathrm{tr}\left( e^{-t\Delta _{B}^{0}}\right)
=\sum_{n\geq 0}e^{-n^{2}t}\sim \frac{\sqrt{\pi }}{2}t^{-1/2}+\frac{1}{2}.
\end{equation*}
Note that the existence of the constant term above implies that the
asymptotics of $K_{B}\left( t\right) $ cannot be obtained by integrating the
asymptotics of $K_{B}\left( t,\theta ,\theta \right) $ above.

We now compute the asymptotics of the trace of the basic heat operator using
Theorem~\ref{codim1nonorientable}. The volume of a generic leaf closure is $%
4\pi$, so the transverse volume is $A_{0}=V_{tr}=\frac{\mathrm{Vol}\left( M
\right)}{4\pi}=\pi$. The mean curvature of the leaf closures is identically
zero, so that $A_{1}=B_{1}=0$. Thus, Theorem~\ref{codim1nonorientable}
implies that 
\begin{equation*}
K_{B}\left( t \right)=\frac{1}{\sqrt{4\pi t}}\left( \pi +\sqrt{\pi}%
\,t^{1/2}+O\left( t^{2} \right) \right) \sim \frac{\sqrt{\pi }}{2}t^{-1/2}+%
\frac{1}{2} ,
\end{equation*}
as expected.
\end{example}

\section{Riemannian Foliations That Are Not Transversally Orientable}

\label{nonorient}

In most of the cases considered throughout this paper, we have assumed that
the foliation in question is transversally oriented. We now remark that with
a few simple modifications, the results of this paper can be used to find
the asymptotics of the trace of the basic heat operator on a Riemannian
foliation that is not transversally orientable. First of all, the basic
Laplacian on functions is still well-defined on such foliations; the basic
Laplacian can defined using a local orientation and does not depend on the
choice of that orientation. Suppose that $\left( M,\mathcal{F}\right) $ is a
Riemannian foliation on a connected, compact manifold with a bundle--like
metric such that the leaves are not transversally orientable. The foliation
may now be lifted to the (nonoriented) orthonormal transverse frame bundle,
an $O\left( q\right) $ bundle over $M$. The lifted foliation is
transversally parallelizable, and the closures of the leaves of the lifted
foliation fiber over a compact $O\left( q\right) $-manifold $\overline{W}$.
\ The group $O\left( q\right) $ does not act by orientation-preserving
isometries, but otherwise the results of this paper extend by letting the
group $G=O\left( q\right) $. Thus, Theorem~\ref{tracebasic} holds, but some
of the results in Section~\ref{special} would have to be modified to allow
for orientation-reversing holonomy. For example, the powers of $t$ in the
asymptotic expansions of the trace of the basic heat operator would in
general increment by half integers instead of integers. Note that this
phenomenon can occur even in the transversally oriented case, if the leaf
closures are not necessarily transversally oriented (see Example~\ref%
{exnonorientable}).

\begin{acknowledgement}
I thank George Gilbert, Efton Park, Franz Kamber, and Jochen Br\"{u}ning for
helpful conversations.
\end{acknowledgement}

\end{document}